\documentclass[11pt,a4paper]{article}
\usepackage[hmargin=3.0cm,vmargin=3.0cm]{geometry}
\usepackage{amssymb,latexsym,amsmath,amsfonts,amsthm}
\usepackage{graphicx}

\usepackage{parskip}
\usepackage{epsfig}
\usepackage{comment}
\usepackage{subfigure}
\usepackage{graphicx,ebezier}
\usepackage{overpic}

\DeclareMathOperator{\diag}{diag} 

\newcommand{\er}{\mathbb{R}}
\newcommand{\cee}{\mathbb{C}}

\newcommand{\lam}{\lambda}

\newcommand{\bol}{\hfill\square\\}

\newcommand{\wtil}{\widetilde}
\newcommand{\what}{\widehat}
\renewcommand{\Re}{\mathrm{Re}\,}

\newcommand{\ct}{D}

\newcommand{\vecA}{\mathbf{A}}

\newcommand{\pee}{\mathbf{p}}
\newcommand{\mm}{\mathbf{m}}
\newcommand{\ud}{\,\mathrm{d}}

\newcommand{\Aa}{\mathcal{A}}

\newcommand{\Ai}{\mathrm{Ai}}
\newcommand{\tac}{_{\mathrm{tac}}}
\newcommand{\pp}{\partial}
\newcommand{\si}{\sigma}
\newcommand{\Kaa}{\mathbf{K}}
\newcommand{\Een}{\mathbf{1}}
\newcommand{\Rop}{\mathbf{R}}

\newtheorem{theorem}{Theorem}[section]
\newtheorem{lemma}[theorem]{Lemma}
\newtheorem{proposition}[theorem]{Proposition}

\newtheorem{corollary}[theorem]{Corollary}

\newtheorem{rhp}[theorem]{RH problem}

\theoremstyle{definition}

\theoremstyle{remark}

\newtheorem{remark}[theorem]{Remark}

\numberwithin{equation}{section}

\title{The tacnode kernel: equality of Riemann-Hilbert and Airy resolvent formulas}
\author{Steven Delvaux\footnotemark[1]}
\date{\today}

\begin{document}

\maketitle
\renewcommand{\thefootnote}{\fnsymbol{footnote}}
\footnotetext[1]{Department of Mathematics, University of Leuven (KU Leuven),
Celestijnenlaan 200B, B-3001 Leuven, Belgium. email:
steven.delvaux\symbol{'100}wis.kuleuven.be. The author is a Postdoctoral Fellow
of the Fund for Scientific Research - Flanders (Belgium). \\
}

\begin{abstract} We study nonintersecting Brownian motions with two prescribed starting and ending
positions, in the neighborhood of a tacnode in the time-space plane. Several
expressions have been obtained in the literature for the critical correlation
kernel $K\tac(x,y)$ that describes the microscopic behavior of the Brownian
motions near the tacnode. One approach, due to Kuijlaars, Zhang and the author,
expresses the kernel (in the single time case) in terms of a $4\times 4$ matrix
valued Riemann-Hilbert problem. Another approach, due to Adler, Ferrari,
Johansson, van Moerbeke and Vet\H o in a series of papers, expresses the kernel
in terms of resolvents and Fredholm determinants of the Airy integral operator
acting on a semi-infinite interval $[\sigma,\infty)$, involving some objects
introduced by Tracy and Widom.

In this paper we prove the equivalence of both approaches. We also obtain a
rank-$2$ property for the derivative of the tacnode kernel. Finally,
we find a Riemann-Hilbert expression for the multi-time extended tacnode
kernel.\smallskip

\end{abstract}

\section{Introduction}

Recently several papers appeared that study $n$ non-intersecting Brownian
motion paths with prescribed starting positions at time $t=0$ and ending
positions at time $t=1$, see
\cite{ADVV,AFvM11,AVV,DelKui2,DKZ,EO,FV,Joh11,KT1,KT3} among many others. If
$n\to\infty$ then the paths fill up a well-defined region in the time-space
plane. By fine-tuning the parameters, we may create a situation with two groups
of Brownian motions, located inside two touching ellipses in the time-space
plane: see the third picture of Figure~\ref{fig:3cases0}. We are interested in
the microscopic behavior of the paths near the touching point of the two
ellipses, i.e., near the tacnode.

\begin{figure}[t]
\begin{center}
\subfigure{\label{figlargesep0}}\includegraphics[scale=0.3]{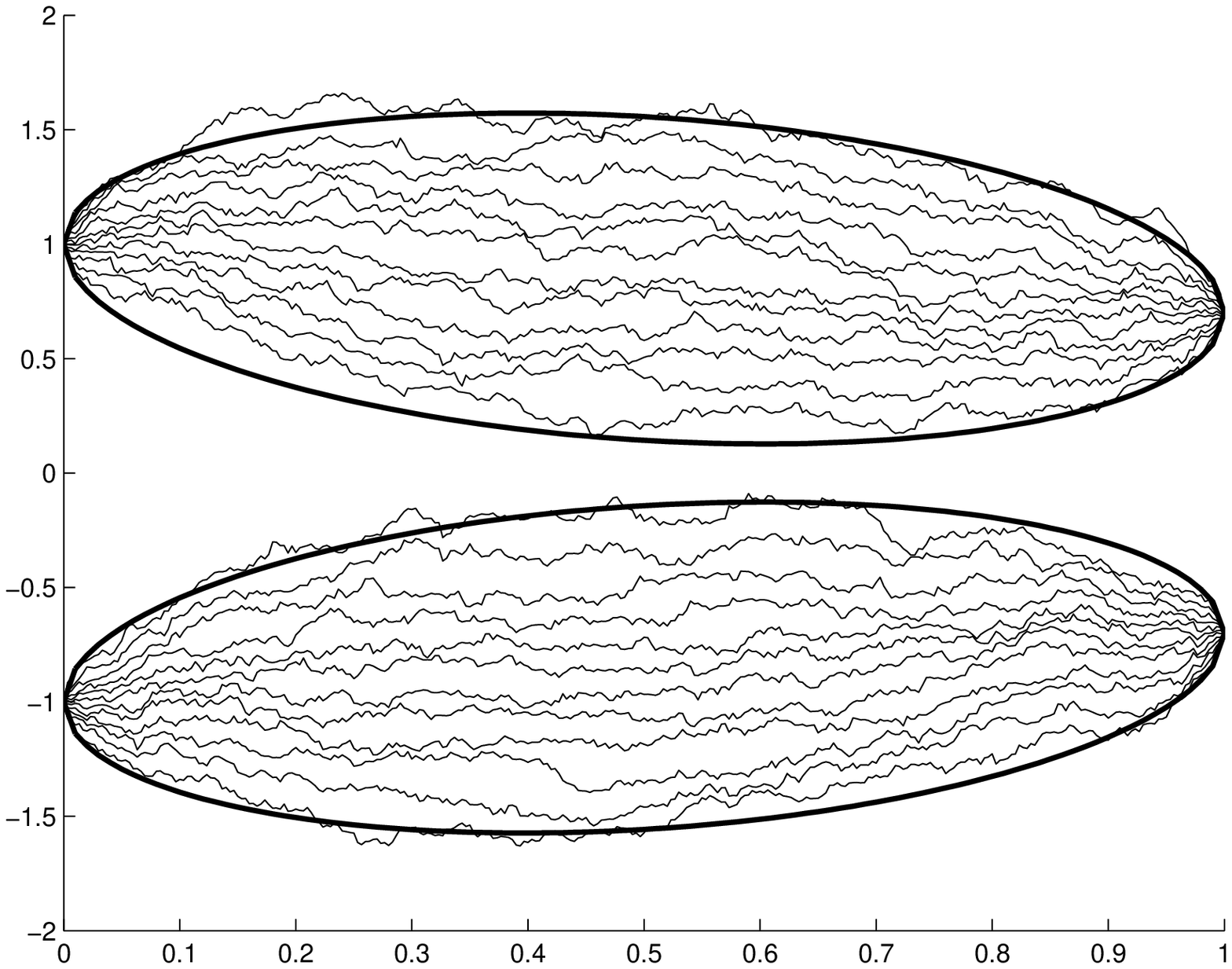}\hspace{5mm}
\subfigure{\label{figsmallsep0}}\includegraphics[scale=0.3]{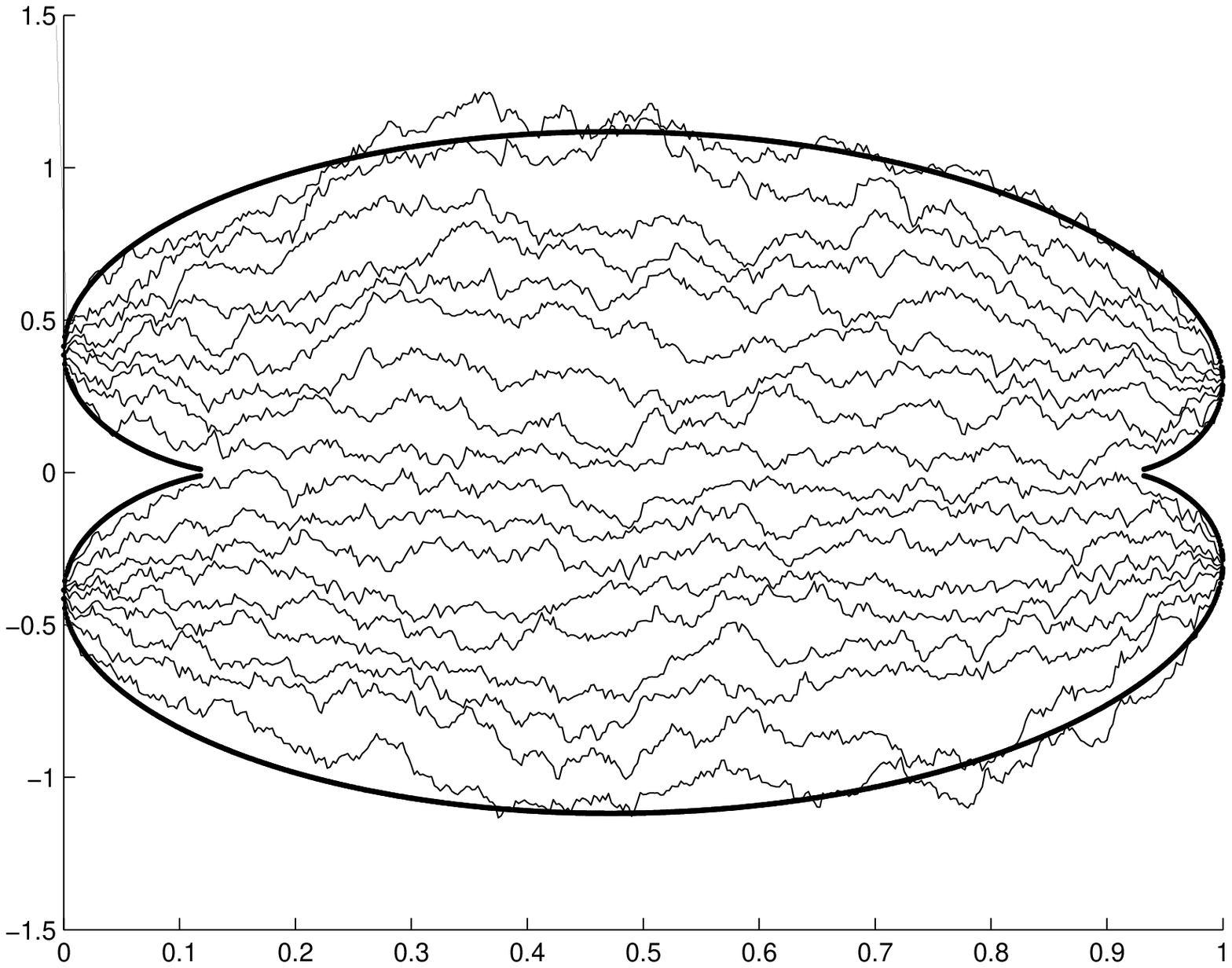}
\subfigure{\label{figcriticalsep0}}\includegraphics[scale=0.3]{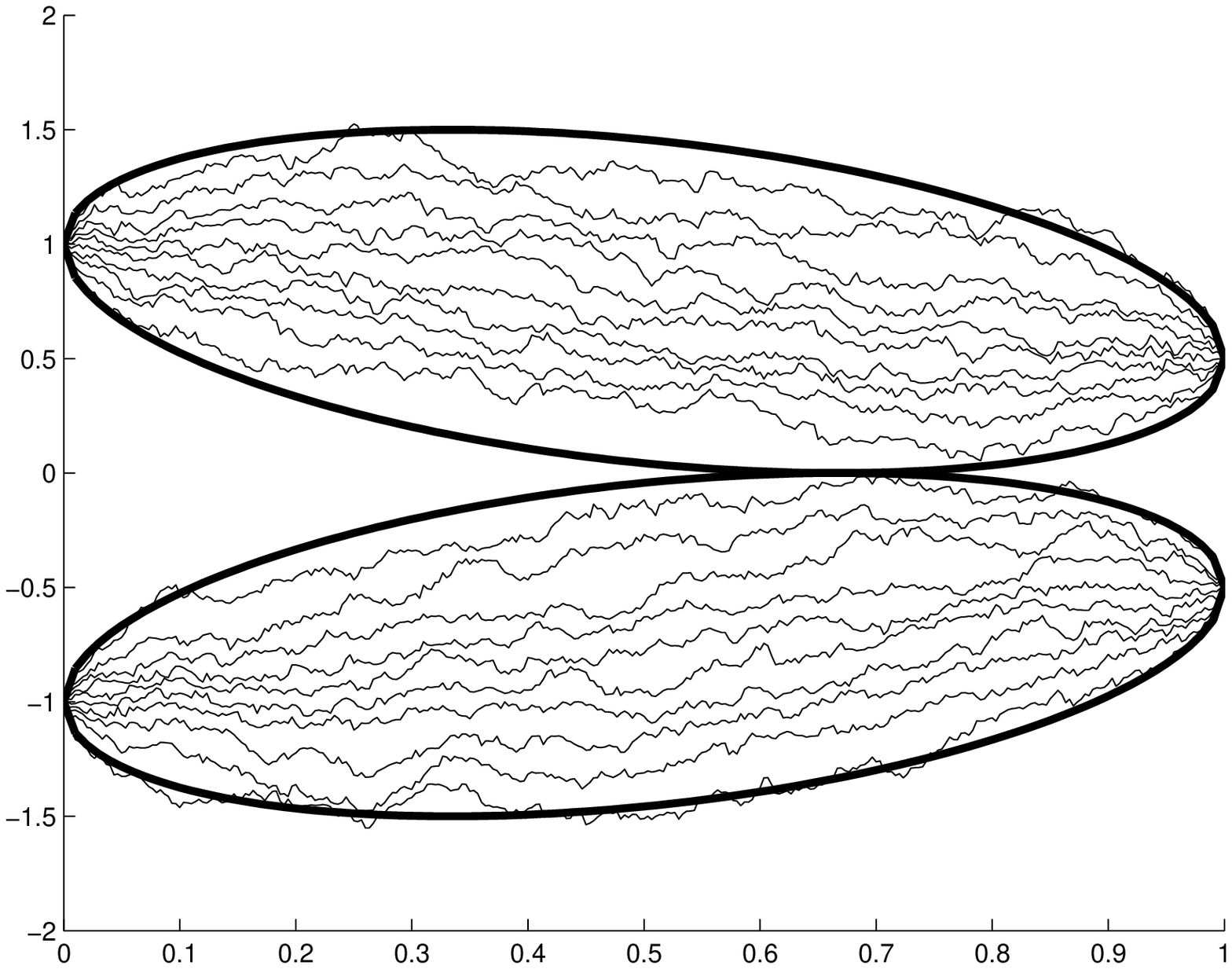}
\end{center}
\caption{$n=20$ non-intersecting Brownian motions at temperature $T=1$ with two
prescribed starting and two ending positions in the case of (a) large, (b)
small, and (c) critical separation between the endpoints. The horizontal axis
stands for the time, $t\in[0,1]$, and the vertical axis shows the positions of
the Brownian motions at time $t$. For $n\to\infty$ the Brownian motions fill a
prescribed region in the time-space plane which is bounded by the boldface
lines in the figure. In the case (c), the limiting support consists of two
touching ellipses which touch each other at a critical point which is a
tacnode.} \label{fig:3cases0}
\end{figure}

It is well-known that the positions of the Brownian motions at any fixed time
$t\in (0,1)$ form a determinantal point process. The process has a well-defined
limit for $n\to\infty$ in a microscopic neighborhood of the tacnode. The
limiting process is encoded by a two-variable correlation kernel $K\tac(x,y)$
which we call the \emph{tacnode kernel}. It depends parametrically on the
scaling that we use near the tacnode. There exists also a multi-time extended
version of the tacnode kernel \cite{AFvM11,AJvM,FV,Joh11}, to be discussed
in Section~\ref{subsection:multitime}.

The tacnode kernel  $K\tac(x,y)$ can be expressed using resolvents and Fredholm
determinants of the Airy integral operator acting on a semi-infinite interval
$[\sigma,\infty)$. This approach was followed in the symmetric case by Adler,
Ferrari, van Moerbeke \cite{AFvM11} and Johansson \cite{Joh11} and in the
non-symmetric case by Ferrari and Vet\H o \cite{FV}. Here the \lq symmetric
case\rq\ means that the two touching groups of Brownian motions have the same
size, as in Figure~\ref{fig:3cases0}, and the non-symmetric case means that one
group is bigger than the other. Similar methods were used to study a double
Aztec diamond \cite{AJvM}, see also \cite{ACJV}.

An alternative expression for the tacnode kernel can be obtained from the
Riemann-Hilbert method. This approach was followed by Kuijlaars, Zhang and the
author in \cite{DKZ}. In that paper we express the tacnode kernel in terms of a
$4\times 4$ matrix-valued Riemann-Hilbert problem (RH problem) $M(z)$, which
yields a new Lax pair representation for the Hastings-McLeod solution $q(x)$ to
the Painlev\'e~II equation. Recall that the \emph{Painlev\'e~II equation} is
the second-order, ordinary differential equation
\begin{equation}\label{def:Painleve2}
    q''(x) = xq(x)+2q^3(x),
\end{equation}
where the prime denotes the derivative with respect to $x$. The
\emph{Hastings-McLeod solution} \cite{FIKN,HML} is the special solution $q(x)$
of \eqref{def:Painleve2} that is real for real $x$ and satisfies
\begin{align}\label{def:HastingMcLeod}
q(x)\sim \Ai(x), & \qquad x\to +\infty,
\end{align}
with $\Ai$ the Airy function.  We note that the usual Riemann-Hilbert matrix
$\Psi(z)$ associated to the Painlev\'e~II equation, due to Flaschka and
Newell~\cite{FN}, has size $2\times 2$ rather than $4\times 4$.

The RH matrix $M(z)$ from the previous paragraph has been the topic of some
recent developments. It was used to study a new critical phenomenon in the
two-matrix model~\cite{DG}, and to establish a reduction from the tacnode
kernel to the Pearcey kernel~\cite{GZ}. It was also extended to a hard-edge
version of the tacnode~\cite{Delvaux2012}.

Summarizing, there exist several, apparently different, formulas for the
tacnode kernel $K\tac(x,y)$. It is natural to ask about the equivalence of
these formulas.

There is an interesting analogy with a model of non-intersecting Brownian
excursions, also known as watermelons (with a wall). The model consists of $n$
Brownian motion paths on the positive half-line with a reflecting or absorbing
wall at the origin. The paths are forced to start and end at the origin and the
interest lies in the maximum position $x_{\max}$ reached by the topmost path
during the time interval $t\in [0,1]$.

Recently, several results were obtained about the joint distribution of the
maximum position $x_{\max}$ and the maximizing time $t_{\max}$ for such
watermelons. One approach, due to Moreno-Quastel-Remenik \cite{MQR}, involves
resolvents and Fredholm determinants of the Airy operator acting on an interval
$[\si,\infty)$. Another approach, due to Schehr \cite{Sch}, involves the
$2\times 2$ Riemann-Hilbert matrix $\Psi(z)$ associated to the Hastings-McLeod
solution to the Painlev\'e~II equation \cite{FN,HML}. The equivalence of both
approaches has been established in a recent work of Baik-Liechty-Schehr
\cite{BLS}. Along the way they obtain Airy resolvent formulas for the entries
of the RH matrix $\Psi(z)$, see also \cite[Sec.~1.1.3]{Baik} for a similar
result.
\\

Inspired by the work of Baik-Liechty-Schehr \cite{BLS}, in this paper we obtain
Airy resolvent formulas for the $4\times 4$ RH matrix $M(z)$. This is the
content of Theorem~\ref{theorem:Airyformulas:FV} below. Our formulas will apply
to the entries in the first and second column of the RH matrix $M(z)$, where
$z$ lies in a sector around the positive imaginary axis.

In Theorem~\ref{theorem:identify:kernels} we will use these formulas to prove
the equivalence of the tacnode kernels in \cite{AFvM11,AJvM,FV,Joh11} and
\cite{DKZ} respectively. We also obtain a remarkable rank-$2$ property for the
derivative of the tacnode kernel, see Theorems~\ref{theorem:tacnodeder} and
\ref{theorem:tacnodeder:FV}.

As a byproduct, the rank-$2$ formula will yield a RH formula for the multi-time
extended tacnode kernel, see Section~\ref{subsection:multitime}. To the best of our knowledge, this is the first time
that a RH formula is obtained for a multi-time extended correlation kernel. See
\cite{BerCaf} for a different connection between RH problems and
multi-time extended point processes, at the level of gap probabilities.

\begin{remark}
Based on an earlier version of this paper, Kuijlaars~\cite{Kui34} has extended
our approach and obtained Airy resolvent formulas for the third and fourth
column of $M(z)$, with $z$ lying again in a sector around the positive
imaginary axis. The latter columns are relevant because they appear in a
critical correlation kernel for the 2-matrix model \cite{DG}.
\end{remark}

\begin{remark}
The paper \cite{Delvaux2012} discusses a hard-edge variant of the tacnode
kernel. The interaction with the hard edge is quantified by a certain parameter
$\alpha>-1$. In the special case $\alpha=0$, the hard edge tacnode kernel
involves the same RH matrix $M(z)$ as the one above and our results give an
alternative way of writing the kernel. It is an open problem to extend our
results to a general value of $\alpha$.
\end{remark}

\section{Statement of results}

\subsection{Definition of the tacnode kernel}

In this section we recall the two different definitions of the tacnode kernel
in the literature. We will denote them by $K\tac(u,v)$ and $\mathcal
L\tac(u,v)$ respectively.

\subsubsection{Definition of the kernel $K\tac$}

A first approach to define the tacnode kernel $K\tac(u,v)$ (in the single-time
case) is via a Riemann-Hilbert problem for a matrix $M(z)$ of size $4\times 4$.

We recall the RH problem from \cite{DKZ,DG}. Fix two numbers
$\varphi_1,\varphi_2$ such that
\begin{equation}\label{def:angles}
0<\varphi_1<\varphi_2<\pi/3.
\end{equation}
Define the half-lines $\Gamma_k$, $k=0,\ldots,9$, by
\begin{equation}\label{def:rays1}
\Gamma_0=\er_+,\quad \Gamma_1=e^{i\varphi_1}\er_+,\quad
\Gamma_2=e^{i\varphi_2}\er_+,\quad \Gamma_3=e^{i(\pi-\varphi_2)}\er_+,\quad
\Gamma_4=e^{i(\pi-\varphi_1)}\er_+,
\end{equation}
and
\begin{equation}\label{def:rays2}
\Gamma_{5+k}=-\Gamma_k,\qquad k=0,\ldots,4.
\end{equation}
All rays $\Gamma_k$, $k=0,\ldots,9$, are oriented towards infinity, as shown in
Figure~\ref{fig:modelRHP}. We denote by $\Omega_k$ the region in $\cee$ that
lies between the rays $\Gamma_k$ and $\Gamma_{k+1}$, for $k=0,\ldots,9$, where
we identify $\Gamma_{10}:=\Gamma_0$.

We consider the following RH problem.

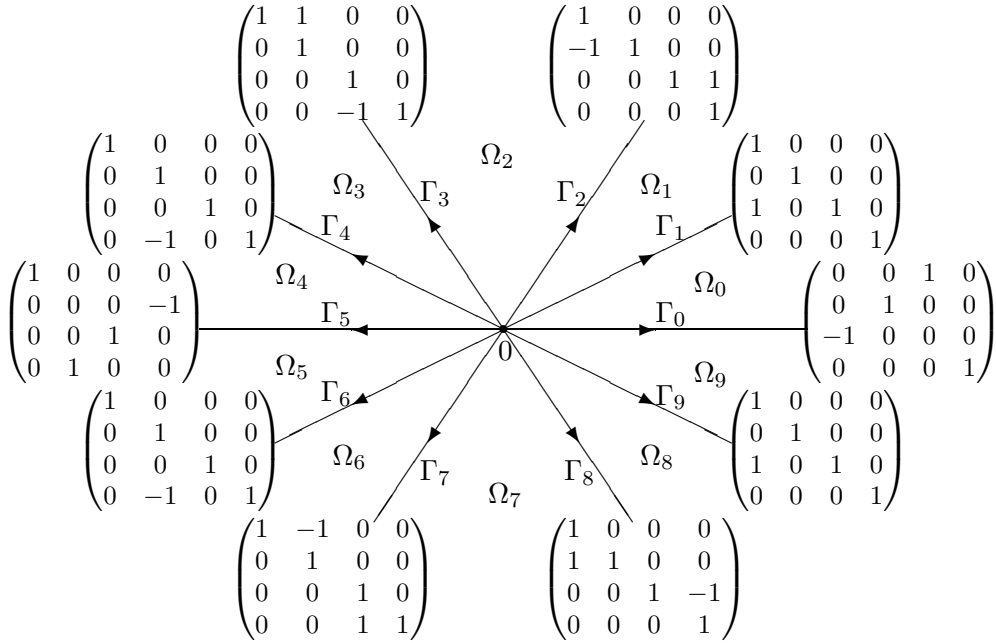
\begin{figure}[t]
\vspace{14mm}
\begin{center}
   \setlength{\unitlength}{1truemm}
   \begin{picture}(100,70)(-5,2)
       \put(40,40){\line(1,0){40}}
       \put(40,40){\line(-1,0){40}}
       \put(40,40){\line(2,1){30}}
       \put(40,40){\line(2,-1){30}}
       \put(40,40){\line(-2,1){30}}
       \put(40,40){\line(-2,-1){30}}
       \put(40,40){\line(2,3){18.5}}
       \put(40,40){\line(2,-3){17}}
       \put(40,40){\line(-2,3){18.5}}
       \put(40,40){\line(-2,-3){17}}
       \put(40,40){\thicklines\circle*{1}}
       \put(39.3,36){$0$}
       \put(60,40){\thicklines\vector(1,0){.0001}}
       \put(20,40){\thicklines\vector(-1,0){.0001}}
       \put(60,50){\thicklines\vector(2,1){.0001}}
       \put(60,30){\thicklines\vector(2,-1){.0001}}
       \put(20,50){\thicklines\vector(-2,1){.0001}}
       \put(20,30){\thicklines\vector(-2,-1){.0001}}
       \put(50,55){\thicklines\vector(2,3){.0001}}
       \put(50,25){\thicklines\vector(2,-3){.0001}}
       \put(30,55){\thicklines\vector(-2,3){.0001}}
       \put(30,25){\thicklines\vector(-2,-3){.0001}}

       \put(60,41){$\Gamma_0$}
       \put(60,52.5){$\Gamma_1$}
       \put(47,57){$\Gamma_2$}
       \put(29,57){$\Gamma_3$}
       \put(16,52.5){$\Gamma_4$}
       \put(16,41){$\Gamma_5$}
       \put(16,30.5){$\Gamma_6$}
       \put(29,20){$\Gamma_7$}
       \put(48,20){$\Gamma_8$}
       \put(60,30){$\Gamma_{9}$}

       \put(65,45){$\small{\Omega_0}$}
       \put(58,58){$\small{\Omega_1}$}
       \put(37,62){$\small{\Omega_2}$}
       \put(17.5,58){$\small{\Omega_3}$}
       \put(10,46){$\small{\Omega_4}$}
       \put(10,34){$\small{\Omega_5}$}
       \put(17.5,22){$\small{\Omega_6}$}
       \put(38,17){$\small{\Omega_7}$}
       \put(58,22){$\small{\Omega_8}$}
       \put(65,33){$\small{\Omega_{9}}$}

       \put(78.5,40){$\small{\begin{pmatrix}0&0&1&0\\ 0&1&0&0\\ -1&0&0&0\\ 0&0&0&1 \end{pmatrix}}$}
       \put(69,57){$\small{\begin{pmatrix}1&0&0&0\\ 0&1&0&0\\ 1&0&1&0\\ 0&0&0&1 \end{pmatrix}}$}
       \put(45,74){$\small{\begin{pmatrix}1&0&0&0\\ -1&1&0&0\\ 0&0&1&1\\ 0&0&0&1 \end{pmatrix}}$}
       \put(4,74){$\small{\begin{pmatrix}1&1&0&0\\ 0&1&0&0\\ 0&0&1&0\\ 0&0&-1&1 \end{pmatrix}}$}
       \put(-16,57){$\small{\begin{pmatrix}1&0&0&0\\ 0&1&0&0\\ 0&0&1&0\\ 0&-1&0&1\end{pmatrix}}$}
       \put(-26,40){$\small{\begin{pmatrix}1&0&0&0\\ 0&0&0&-1\\ 0&0&1&0\\ 0&1&0&0 \end{pmatrix}}$}
       \put(-16,23){$\small{\begin{pmatrix}1&0&0&0\\ 0&1&0&0\\ 0&0&1&0\\ 0&-1&0&1 \end{pmatrix}}$}
       \put(4,6){$\small{\begin{pmatrix}1&-1&0&0\\ 0&1&0&0\\ 0&0&1&0\\ 0&0&1&1 \end{pmatrix}}$}
       \put(45,6){$\small{\begin{pmatrix}1&0&0&0\\ 1&1&0&0\\ 0&0&1&-1\\ 0&0&0&1 \end{pmatrix}}$}
       \put(69,23){$\small{\begin{pmatrix}1&0&0&0\\ 0&1&0&0\\ 1&0&1&0\\ 0&0&0&1\end{pmatrix}}$}

  \end{picture}
   \vspace{0mm}
   \caption{The figure shows the jump contours $\Gamma_k$ in the complex $z$-plane and the corresponding jump
   matrix
   $J_k$ on $\Gamma_k$, $k=0,\ldots,9$, in the RH problem for  $M = M(z)$. We denote by $\Omega_k$ the region between
    the rays $\Gamma_k$ and $\Gamma_{k+1}$.}
   \label{fig:modelRHP}
\end{center}
\end{figure}

\begin{rhp}\label{rhp:modelM} We look for a matrix valued function
$M:\cee\setminus\left(\bigcup_{k=0}^{9}\Gamma_k\right)\to\cee^{4\times 4}$
(which also depends on the parameters $r_1,r_2>0$ and $s_1,s_2,\tau\in\cee$)
satisfying
\begin{itemize}
\item[(1)] $M(z)$ is analytic (entrywise) for $z\in\cee\setminus\left(\bigcup_{k=0}^{9}
\Gamma_k\right)$.
\item[(2)] For $z\in\Gamma_k$, the limiting values
\[ M_+(z) = \lim_{x \to z, \,  x\textrm{ on $+$-side of }\Gamma_k} M(x), \qquad
    M_-(z) = \lim_{x \to z, \, x\textrm{ on $-$-side of }\Gamma_k} M(x) \]
exist, where the $+$-side and $-$-side of $\Gamma_k$ are the sides which lie on
the left and right of $\Gamma_k$, respectively, when traversing $\Gamma_k$
according to its orientation. These limiting values satisfy the jump relation
\begin{equation}\label{jumps:M}
M_{+}(z) = M_{-}(z)J_k(z),\qquad k=0,\ldots,9,
\end{equation}
where the jump matrix $J_k(z)$ for each ray $\Gamma_k$ is shown in
Figure~\ref{fig:modelRHP}.
\item[(3)] As $z\to\infty$ we have
\begin{multline}
\label{M:asymptotics} M(z) =
\left(I+\frac{M_1}{z}+\frac{M_2}{z^2}+O\left(\frac{1}{z^3}\right)\right)
\diag((-z)^{-1/4},z^{-1/4},(-z)^{1/4},z^{1/4})
\\
\times
\Aa\diag\left(e^{\theta_1(z)},e^{\theta_2(z)},e^{\theta_3(z)},e^{\theta_4(z)}\right),
\end{multline}
where the coefficient matrices $M_1,M_2,\ldots$ are independent of $z$, and
with
\begin{equation}\label{mixing:matrix}
\Aa:=\frac{1}{\sqrt{2}}\begin{pmatrix} 1 & 0 & -i & 0 \\
0 & 1 & 0 & i \\
-i & 0 & 1 & 0 \\
0 & i & 0 & 1 \\
\end{pmatrix},
\end{equation}
and
\begin{equation}\label{def:theta1}
\left\{\begin{array}{lll} \theta_1(z) =
-\frac{2}{3}r_1(-z)^{3/2}&\hspace{-3mm}-2s_1(-z)^{1/2}&\hspace{-3mm}+r_1^2\tau z, \\
\theta_2(z) = -\frac{2}{3}r_2z^{3/2}&\hspace{-3mm}-2s_2z^{1/2}&\hspace{-3mm}-r_2^2\tau z,\\
\theta_3(z) =
\hphantom{-}\frac{2}{3}r_1(-z)^{3/2}&\hspace{-3mm}+2s_1(-z)^{1/2}&\hspace{-3mm}+r_1^2\tau z, \\
\theta_4(z) =
\hphantom{-}\frac{2}{3}r_2z^{3/2}&\hspace{-3mm}+2s_2z^{1/2}&\hspace{-3mm}-r_2^2\tau
z.\end{array}\right.
\end{equation}
Here we use the principal branches of the fractional powers.
\item[(4)] $M(z)$ is bounded as $z\to 0$.
\end{itemize}
\end{rhp}

We will sometimes write $M(z)=M(z;r_1,r_2,s_1,s_2,\tau)$ to indicate the
dependence on the parameters. The factors $r_1^2$ and $r_2^2$ in front of $\tau z$ in \eqref{def:theta1}
will be useful in the statement of our main theorems; these factors could be removed by a scaling and translation
of the RH matrix. We could also assume $r_2=1$ without loss of
generality by a simple rescaling of $z$.

The RH problem~\ref{rhp:modelM} was introduced in \cite{DKZ} with $\tau=0$ in
\eqref{def:theta1}. The parameter $\tau$ was introduced in \cite{DG} in the
symmetric setting where $r_1=r_2=1$ and $s_1=s_2$. The general non-symmetric
case with the extra parameter $\tau$ has not been considered before in the
literature. The following result can be proved as in \cite{DG}; see also the
discussion following Lemma~\ref{lemma:DG:2} below.

\begin{proposition}\label{prop:solv} (Solvability.)
For any $r_1,r_2>0$ and $s_1,s_2,\tau\in\er$ there is a unique solution
$M(z)=M(z;r_1,r_2,s_1,s_2,\tau)$ to the RH problem~\ref{rhp:modelM}.
\end{proposition}

Now we define the tacnode kernel. Let $r_1,r_2>0$ and $s_1,s_2,\tau\in\er$ be
fixed parameters. Let $\what M(z)$ be the restriction of $M(z)$ to the sector
$z\in\Omega_2$ around the positive imaginary axis. We extend $\what M(z)$ to
the whole complex $z$-plane by analytic continuation. This analytic
continuation is well-defined, since the product $J_3J_4\cdots J_9J_0J_1J_2$ of
the jump matrices in the RH problem~\ref{rhp:modelM} is the identity matrix.

The tacnode kernel $K\tac(u,v)$ is defined in terms of the RH matrix $\what
M(z)$ by \cite[Def.~2.6]{DKZ}\footnote{We note that \cite[Def.~2.6]{DKZ} has a
typo: it has \lq $\what M^{-1}(u) \what M(v)$\rq\ instead of \lq $\what
M^{-1}(v) \what M(u)$\rq.}
\begin{equation}\label{tacnode:DKZ} K\tac(u,v) = \frac{1}{2\pi i(u-v)}
\begin{pmatrix} 0&0&1&1\end{pmatrix}
\what M^{-1}(v) \what M(u)
\begin{pmatrix} 1\\1\\0\\0\end{pmatrix}.
\end{equation}

For later use, it is convenient to denote by
\begin{equation}\label{pee:firsttwocolumns}\pee(z) = \what
M(z)\begin{pmatrix}1\\1\\0\\0\end{pmatrix}\in\cee^{4\times 1} \end{equation} the
sum of the first and second column of $\what M(z)$. Observe that
\eqref{tacnode:DKZ} and \eqref{pee:firsttwocolumns} both depend on the
parameters $r_1,r_2,s_1,s_2,\tau$.

\subsubsection{Definition of the kernel $\mathcal L\tac$}

In this section we recall the second way to define the tacnode kernel, via Airy
resolvents \cite{AFvM11,AJvM,FV,Joh11}. We will denote this kernel by $\mathcal
L\tac(u,v)$.

Denote by $\Ai(x)$ the standard Airy function and by
\begin{equation}\label{Airy:kernel:0} K_{\Ai}(x,y)
= \frac{\Ai(x)\Ai'(y)-\Ai'(x)\Ai(y)}{x-y}=\int_{0}^{\infty} \Ai(x+z)\Ai(y+z)\ud
z
\end{equation}
the Airy kernel. For $\si\in\er$ let
\begin{equation}\label{Airy:kernel} K_{\Ai,\si}(x,y)
= \int_{0}^{\infty} \Ai(x+z+\si)\Ai(y+z+\si)\ud z
\end{equation}
be the Airy kernel shifted by $\si$. Let $\Kaa_{\Ai,\si}$ be the integral
operator with kernel $K_{\Ai,\si}$ acting on the function space
$L^2([0,\infty))$.
The action of the operator $\Kaa_{\Ai,\si}$ on the function $f$ is defined by
$$ [\Kaa_{\Ai,\si} f]\ (x) = \int_{0}^{\infty} K_{\Ai,\si}(x,y)f(y)\ud y.
$$
Define the resolvent operator $\Rop_{\Ai,\si}$ on $L^2([0,\infty))$ by
\begin{equation}\label{resolvent:def} \Rop_{\Ai,\si} := (\Een-\Kaa_{\Ai,\si})^{-1} - \Een =
\Kaa_{\Ai,\si}(\Een-\Kaa_{\Ai,\si})^{-1} =
(\Een-\Kaa_{\Ai,\si})^{-1}\Kaa_{\Ai,\si},
\end{equation}
where $\Een$ stands for the identity operator on $L^2([0,\infty))$. It is known
that $\Rop_{\Ai,\si}$ is again an integral operator on $L^2([0,\infty))$ and we
denote its kernel by $R_{\si}(x,y)$:
\begin{equation}\label{resolvent:Rxy}
[\Rop_{\Ai,\si} f]\ (x) = \int_{0}^{\infty} R_{\si}(x,y)f(y)\ud y.
\end{equation}
We will sometimes use the notation
\begin{equation}\label{Dirac:delta}
(\Een-\Kaa_{\Ai,\si})^{-1}(x,y) \equiv (\Een+\Rop_{\Ai,\si})(x,y) :=
\delta(x-y)+R_{\si}(x,y),
\end{equation}
with $\delta(x-y)$ the Dirac delta function at $x=y$ and $R_{\si}$ the Airy
resolvent kernel \eqref{resolvent:Rxy}. We will often use the symmetry of the
kernel, $R_{\si}(x,y)=R_{\si}(y,x)$. Finally, we will abbreviate $R_{\si}$ by
$R$ if the value of $\si$ is clear from the context.
\\

In a series of papers \cite{AFvM11,AJvM,FV,Joh11}, Adler, Ferrari, Johansson,
van Moerbeke and Vet\H o study the tacnode problem using Airy resolvent
expressions. We focus in particular on the paper by Ferrari-Vet\H o \cite{FV}
on the non-symmetric tacnode. Hence the two touching groups of Brownian motions
at the tacnode are allowed to have a different size. The paper \cite{FV} uses a
parameter $\lam>0$ that quantifies the amount of asymmetry, with $\lam=1$
corresponding to the symmetric case treated in \cite{AFvM11,AJvM,Joh11}. These
papers also use a parameter $\si>0$ that controls the strength of interaction
between the two groups of Brownian motions near the tacnode. In the present
paper we will denote the latter parameter with a capital $\Sigma$. The
parameter $\Sigma$ has a similar effect on the tacnode kernel as the
temperature parameter used in \cite{Delvaux2012,DelKui2} (suitably rescaled).
In order to be consistent with \cite{DKZ}, we will use the notation $\si$ to
denote
\begin{equation}\label{temperature:sigma} \si  = \lam^{1/2}(1+\lam^{-1/2})^{2/3}\Sigma.
\end{equation}
(What we call $\si$ was called $\wtil\si$ in \cite{AFvM11,AJvM,FV,Joh11}.) The
papers \cite{AFvM11,AJvM,FV,Joh11} consider a multi-time extended tacnode
kernel with time variables $\tau_1,\tau_2$. We restrict ourselves here to the
single time case $\tau_1=\tau_2=:\tau$. The discussion of the multi-time case is
postponed to Section~\ref{subsection:multitime}.

With the above notations, define the functions \cite{FV}\footnote{The notations
$b_{\tau,\si+\xi}^{\lam}(x+\wtil\si)$ and
$b_{\lam^{1/3}\tau,\lam^{2/3}\si-\lam^{1/6}\xi}^{\lam^{-1}}(x+\wtil\si)$ in
\cite{FV} correspond to our notations $\lam^{1/6}\wtil b_{-\tau,\xi}(x)$ and
$\lam^{-1/6}b_{-\tau,\xi}(x)$ respectively.}
\begin{equation}
\label{b:FV} \begin{array}{ll} b_{\tau,z}(x) = \exp\left(-\tau
y+\tau^3/3\right)\Ai(y), &\textrm{with }y:=z+Cx+\Sigma+\tau^2,
\\
\wtil b_{\tau,z}(x) = \exp\left(-\sqrt{\lam}\tau \wtil
y+\lam\tau^3/3\right)\Ai(\lam^{1/6}\wtil y),&\textrm{with }\wtil
y:=-z+Cx+\sqrt{\lam}(\Sigma+\tau^2),\end{array}
\end{equation}
where \begin{equation} \label{C:FV0} C = (1+\lam^{-1/2})^{1/3}.
\end{equation}
Note that in the symmetric case $\lam=1$, we have $\wtil
b_{\tau,z}(x)=b_{\tau,-z}(x)$. The functions \eqref{b:FV} also depend on
$\lam,\si$ (recall \eqref{temperature:sigma}) but we do not show this in the
notation. Next, we define the functions
\begin{equation}\label{Acal:FV} \begin{array}{l} \displaystyle\Aa_{\tau,z}(x) = b_{\tau,z}(x)
-\lam^{1/6}\int_{0}^{\infty} \Ai(x+y+\si)\wtil b_{\tau,z}(y) \ud y
\\ \displaystyle \wtil\Aa_{\tau,z}(x) =
\wtil b_{\tau,z}(x)- \lam^{-1/6}\int_{0}^{\infty} \Ai(x+y+\si)b_{\tau,z}(y)\ud
y.\end{array}
\end{equation}
Again we have $\wtil\Aa_{\tau,z}(x) = \Aa_{\tau,-z}(x)$ in the symmetric case
$\lam=1$, and we suppress the dependence on $\lam,\si$ from the notation.\\

We are now ready to introduce the tacnode kernel $\mathcal L\tac(u,v)=\mathcal
L\tac(u,v;\si,\tau)$ of Ferrari-Vet\H o \cite{FV}\footnote{This kernel is
called $\mathcal L\tac^{\lam,\si}(\tau_1,\xi_1,\tau_2,\xi_2)$ in \cite{FV},
with $\xi_1=u$, $\xi_2=v$ and $\tau_1=\tau_2=\tau$. Recall that we use $\si$ in
a different meaning.}, restricted to the single-time case $\tau_1=\tau_2=\tau$.
The kernel can be represented in several equivalent ways. We find it convenient
to use the following representation.

\begin{proposition}\label{prop:FVkernel} (The tacnode kernel.)
Fix $\lam>0$. The tacnode kernel $\mathcal L\tac(u,v)$ of Ferrari-Vet\H o
\cite{FV} in the single-time case $\tau_1=\tau_2=\tau$ can be written in the
form
\begin{multline}\label{tacnode:FV}
\mathcal L\tac(u,v;\si,\tau) =
C\lam^{1/3}\int_0^{\infty}\wtil b_{\tau,u}(x)\wtil b_{-\tau,v}(x)\ud x
 \\
+C\int_0^{\infty}\!\!\!\int_0^{\infty}(\Een-\Kaa_{\Ai,\si})^{-1}(x,y)
\Aa_{\tau,u}(x)\Aa_{-\tau,v}(y)\ud x\ud y,
\end{multline}
with the notations \eqref{Airy:kernel} and
\eqref{Dirac:delta}--\eqref{Acal:FV}.
\end{proposition}

Proposition~\ref{prop:FVkernel}  is proved in
Section~\ref{subsection:proof:FVkernel} and its multi-time extended version is stated in Section~\ref{subsection:multitime}. The proposition was obtained in the symmetric case
$\lam=1$ by Adler, Johansson and van Moerbeke \cite[Theorem~1.2(i)]{AJvM}.

Incidentally we note
that a kernel of the form \eqref{tacnode:FV} allows an efficient numerical evaluation of its gap probabilities,
as is recently shown in \cite{BerCaf2}.

In the above proposition we use the notation \eqref{Dirac:delta}. Hence the
double integral can be rewritten as
\begin{multline*}
\int_0^{\infty}\!\!\!\int_0^{\infty}(\Een-\Kaa_{\Ai,\si})^{-1}(x,y)\,
\Aa_{\tau,u}(x)\Aa_{-\tau,v}(y)\ud x\ud y \\
= \int_0^{\infty}\!\!\!\int_0^{\infty}R_{\si}(x,y)
\Aa_{\tau,u}(x)\Aa_{-\tau,v}(y)\ud x\ud y+\int_0^{\infty}
\Aa_{\tau,u}(x)\Aa_{-\tau,v}(x)\ud x.
\end{multline*}

\subsection{Connection between the tacnode kernels $K\tac$ and $\mathcal L\tac$}

Now we state the first main theorem of this paper.

\begin{theorem}\label{theorem:identify:kernels} (Connection between kernels.)
Fix $\lam>0$. The kernel $\mathcal L\tac(u,v;\si,\tau)$ in \eqref{tacnode:FV}
equals the RH kernel $K\tac(u,v;r_1,r_2,s_1,s_2,\tau)$ in \eqref{tacnode:DKZ}
with the parameters
\begin{equation}\label{identify:kernels} r_1=\lam^{1/4},\qquad r_2=1,\qquad
s_1=\frac{1}{2}\lam^{3/4}(\Sigma+\tau^2), \qquad
s_2=\frac{1}{2}(\Sigma+\tau^2),
\end{equation}
where we recall \eqref{temperature:sigma}.
\end{theorem}

Theorem~\ref{theorem:identify:kernels} is proved in
Section~\ref{subsection:proof:identify}.

\subsection{Derivative of the tacnode kernel: a rank $2$ property}

To prove Theorem~\ref{theorem:identify:kernels} we will take the derivatives of
the kernels $K\tac$ and $\mathcal L\tac$ with respect to a certain parameter,
and prove that they are equal. The derivative will be a rank $2$ kernel. First
we discuss this for the kernel $K\tac$.

\subsubsection{Derivative of the kernel $K\tac$}

Recall the formula \eqref{tacnode:DKZ} for the tacnode kernel $K\tac$. This
kernel has an \lq integrable\rq\ form, due to the factor $u-v$ in the
denominator. Interestingly, this factor cancels when taking the derivative with
respect to $s_1$ or $s_2$. This is the content of the next theorem.

To state the theorem we parameterize
\begin{equation}\label{s12:s}
s_1 =: \sigma_1 s,\qquad s_2 =: \sigma_2 s,
\end{equation}
where $\sigma_1,\sigma_2$ are fixed and $s$ is variable. In the symmetric case
where $s_1=s_2=:s$, we could simply take $\si_1=\si_2=1$. We also consider
$r_1,r_2>0$ to be fixed. Then we write $K\tac(u,v;s,\tau)$, $\pee(z;s,\tau)$
etc., to denote the dependence on the two parameters $s$ and $\tau$.

\begin{theorem}\label{theorem:tacnodeder} (Derivative of tacnode kernel $K\tac$.)
With the parametrization \eqref{s12:s}, the kernel \eqref{tacnode:DKZ}
satisfies
\begin{equation}\label{rank2:tacnode:RH} \frac{\pp}{\pp s}
K\tac(u,v;s,\tau) = -\frac{1}{\pi}\left(\sigma_1
p_1(u;s,\tau)p_1(v;s,-\tau)+\sigma_2 p_2(u;s,\tau)p_2(v;s,-\tau)\right),
\end{equation}
where $p_j$, $j=1,\ldots,4$, denotes the $j$th entry of the vector $\pee$ in
\eqref{pee:firsttwocolumns}. Consequently, if $\si_1,\si_2>0$ then
\begin{equation}\label{integrated:kernel}
K\tac(u,v;s,\tau) = \frac{1}{\pi}\int_s^{\infty}\left(\sigma_1 p_1(u;\wtil
s,\tau)p_1(v;\wtil s,-\tau)+\sigma_2 p_2(u;\wtil s,\tau)p_2(v;\wtil
s,-\tau)\right)\ud \wtil s.
\end{equation}
\end{theorem}

Theorem~\ref{theorem:tacnodeder} is proved in
Section~\ref{subsection:proof:tacnodeder}. Note that the right hand
side of \eqref{rank2:tacnode:RH} is a rank-$2$ kernel. Later we will apply the theorem with
the parameters $\si_1,\si_2$ in \eqref{identify:s12}.

\begin{remark}
Formulas \eqref{rank2:tacnode:RH}--\eqref{integrated:kernel} have an analogue
for the kernel $K_{\Psi}(u,v)$ which is associated to the $2\times 2$
Flaschka-Newell RH matrix $\Psi(z)$, see \cite[Eq.~(1.22)]{CK}. The kernel
$K_{\Psi}(u,v)$ occurs in Hermitian random matrix theory when the limiting
eigenvalue density vanishes quadratically at an interior point of its support
\cite{BI,CK}.
\end{remark}

\subsubsection{Derivative of the kernel $\mathcal L\tac$}
\label{subsection:tacnodeder:Airy}

Next we consider the derivative of the tacnode kernel $\mathcal L\tac(u,v;\si,\tau)$
with respect to the parameter $\si$.

\begin{theorem}\label{theorem:tacnodeder:FV} (Derivative of tacnode kernel $\mathcal L\tac$.) Fix $\lam>0$. The kernel
\eqref{tacnode:FV} satisfies
\begin{equation}\label{rank2:tacnode:FV} \frac{\pp}{\pp \si}
\mathcal L\tac(u,v;\si,\tau) = -C^{-2}\left(\lam^{1/3}\what
p_1(u;\si,\tau)\what p_1(v;\si,-\tau)+\lam^{-1/2}\what p_2(u;\si,\tau)\what
p_2(v;\si,-\tau)\right)
\end{equation}
and consequently
\begin{equation}\label{integrated:tacnode:FV}
\mathcal L\tac(u,v;\si,\tau) = C^{-2}\int_{\si}^{\infty}\left(\lam^{1/3}\what
p_1(u;s,\tau)\what p_1(v;s,-\tau)+\lam^{-1/2}\what p_2(u;s,\tau)\what
p_2(v;s,-\tau)\right)\ud s.
\end{equation}
Here we denote $C=(1+\lam^{-1/2})^{1/3}$ and
\begin{equation}
\label{p1hat:op} \begin{array}{l}\displaystyle \what p_1(z;\si,\tau) =
\int_0^{\infty}(\Een-\Kaa_{\Ai,\si})^{-1}(x,0)\,\wtil\Aa_{\tau,z}(x)\ud x
\\ \displaystyle \what p_2(z;\si,\tau) =
\int_0^{\infty}(\Een-\Kaa_{\Ai,\si})^{-1}(x,0)\,\Aa_{\tau,z}(x)\ud
x,\end{array}
\end{equation}
with the notation \eqref{Dirac:delta}.
\end{theorem}

Theorem~\ref{theorem:tacnodeder:FV} is proved in
Section~\ref{subsection:proof:tacnodeder:FV} and its multi-time extended version is stated in Section~\ref{subsection:multitime}. In the proof we will use certain
functions \eqref{Q:Airy:def} that were introduced by Tracy-Widom \cite{TW1}. We
will also obtain some alternative representations for \eqref{p1hat:op}, see
Lemma~\ref{lemma:equiv:ways}.

Note that the right hand side of \eqref{rank2:tacnode:FV} is again a rank-$2$
kernel. This is a result of independent interest.

\subsection{The multi-time case}
\label{subsection:multitime}

Our results for the Ferrari-Vet\H o tacnode kernel can be readily generalized to the multi-time case, corresponding to two different times $\tau_1,\tau_2$. It suffices to replace all the subscripts $\tau$ and $-\tau$, by $\tau_1$ and $-\tau_2$ respectively. This yields the following generalization of Proposition~\ref{prop:FVkernel}.

\begin{proposition}\label{prop:FVkernel:multi} (Extended tacnode kernel.)
Fix $\lam>0$. The multi-time extended tacnode kernel $\mathcal L\tac$ of Ferrari-Vet\H o
\cite{FV} can be written in the
form\footnote{This kernel is
called $\mathcal L\tac^{\lam,\si}(\tau_1,\xi_1,\tau_2,\xi_2)$ in \cite{FV} with $\xi_1=u$ and $\xi_2=v$. Recall that we use $\si$ in
a different meaning.}
\begin{multline}\label{tacnode:FV:multi}
\mathcal L\tac(u,v;\si,\tau_1,\tau_2) = -\mathbf{1}_{\tau_1<\tau_2}\frac{1}{\sqrt{4\pi(\tau_2-\tau_1)}}\exp\left(-\frac{(v-u)^2}{4(\tau_2-\tau_1)}\right)
\\+C\lam^{1/3}\int_0^{\infty}\wtil b_{\tau_1,u}(x)\wtil b_{-\tau_2,v}(x)\ud x
+C\int_0^{\infty}\!\!\!\int_0^{\infty}(\Een-\Kaa_{\Ai,\si})^{-1}(x,y)
\Aa_{\tau_1,u}(x)\Aa_{-\tau_2,v}(y)\ud x\ud y,
\end{multline}
with the notations \eqref{Airy:kernel} and
\eqref{Dirac:delta}--\eqref{Acal:FV}.
\end{proposition}

The rank $2$ formula in Theorem~\ref{theorem:tacnodeder:FV} generalizes as follows.

\begin{theorem}\label{theorem:tacnodeder:FV:multi} (Derivative of extended tacnode kernel.) Fix $\lam>0$. The kernel
\eqref{tacnode:FV:multi} satisfies
\begin{equation}\label{rank2:tacnode:FV:multi} \frac{\pp}{\pp \si}
\mathcal L\tac(u,v;\si,\tau_1,\tau_2) = -C^{-2}\left(\lam^{1/3}\what
p_1(u;\si,\tau_1)\what p_1(v;\si,-\tau_2)+\lam^{-1/2}\what p_2(u;\si,\tau_1)\what
p_2(v;\si,-\tau_2)\right)
\end{equation}
and consequently
\begin{multline}\label{integrated:tacnode:FV:multi}
\mathcal L\tac(u,v;\si,\tau_1,\tau_2) = -\mathbf{1}_{\tau_1<\tau_2}\frac{1}{\sqrt{4\pi(\tau_2-\tau_1)}}\exp\left(-\frac{(v-u)^2}{4(\tau_2-\tau_1)}\right)
\\ +C^{-2}\int_{\si}^{\infty}\left(\lam^{1/3}\what
p_1(u;s,\tau_1)\what p_1(v;s,-\tau_2)+\lam^{-1/2}\what p_2(u;s,\tau_1)\what
p_2(v;s,-\tau_2)\right)\ud s.
\end{multline}
Here we use again the notations $C=(1+\lam^{-1/2})^{1/3}$  and \eqref{p1hat:op}.
\end{theorem}

As an offshoot of this theorem, we obtain a Riemann-Hilbert expression for the multi-time extended tacnode kernel.
Indeed the functions $\what p_1$ and $\what p_2$ can be expressed in terms of the top left $2\times 2$ block of the RH matrix $M(z)$ on account of
\eqref{pee:firsttwocolumns}, \eqref{identify:kernels} and \eqref{pphat}. To the best of our knowledge, this is the first time that a RH expression is given for a multi-time extended kernel.

The above results on the multi-time extended tacnode kernel can be proved exactly as in the single-time case. We omit the details. In the remainder of this paper we will not come back on the multi-time case anymore.

\subsection{Airy resolvent formulas for the $4\times 4$ Riemann-Hilbert matrix}

In the proof of Theorem~\ref{theorem:identify:kernels} we will need Airy
resolvent formulas for the entries in the first two columns of the RH matrix
$\what M(z) = \what M(z;r_1,r_2,s_1,s_2,\tau)$. The existence of such formulas
could be anticipated by comparing Theorems \ref{theorem:tacnodeder} and
\ref{theorem:tacnodeder:FV}.

The formulas below will be stated for general values of $r_1,r_2,s_1,s_2,\tau$.
The reader who is only interested in the symmetric case $r_1=r_2$ and $s_1=s_2$
can skip the next paragraph and move directly to \eqref{sigma:symm}.

For general $r_1,r_2>0$ and $s_1,s_2,\tau\in\er$, we define the constants
\begin{eqnarray} \nonumber C &=& (r_1^{-2}+r_2^{-2})^{1/3},
\\ \label{C:RH} \ct &=&
\sqrt{\frac{r_1}{r_2}}\,\exp\left(\frac{r_1^4-r_2^4}{3}\tau^3+2(r_2s_2-r_1s_1)\tau\right),\\
 \nonumber \si  &=& C^{-1}\left(2\left(\frac{s_1}{r_1}+\frac{s_2}{r_2}\right)-(r_1^2+r_2^2)\tau^2\right),
\end{eqnarray} and the functions
\begin{equation}\label{bzx:FV} \begin{array}{l} \displaystyle b_z(x) =
\sqrt{2\pi}r_2^{1/6}\exp\left(-r_2^2 \tau \left(z+C
x\right)\right)\Ai\left(r_2^{2/3}\left(z+Cx+2\frac{s_2}{r_2}\right)\right),
\\
\displaystyle \wtil b_z(x) = \sqrt{2\pi}r_1^{1/6}\exp\left(r_1^2 \tau \left(z-C
x\right)\right)\Ai\left(r_1^{2/3}\left(-z+Cx+2\frac{s_1}{r_1}\right)\right).\end{array}
\end{equation}

The above definitions of $C,\si$ are consistent with our earlier formulas
\eqref{C:FV0} and \eqref{temperature:sigma} under the identification
\eqref{identify:kernels}. Similarly, the formulas for $b_z(x),\wtil b_z(x)$
reduce to the ones in \eqref{b:FV} up to certain multiplicative constants
(independent of $z,x$).

In the symmetric case where $r_1=r_2=1$,
$s_1=s_2=:s$, the above definitions simplify to \begin{eqnarray}\nonumber C&=&2^{1/3},\\
\nonumber\ct &=&1,\\ \label{sigma:symm}\si  &=& 2^{5/3}s -2^{2/3}\tau^2, \\
\nonumber b_z(x) &=& \sqrt{2\pi}\exp\left(-\tau \left(z+2^{1/3}
x\right)\right)\Ai\left(z+2^{1/3}x+2s\right),\\ \nonumber\wtil
b_z(x)&=&b_{-z}(x).\end{eqnarray}

The fact of the matter is

\begin{theorem}\label{theorem:Airyformulas:FV} (Airy resolvent formulas for the $4\times 4$ RH matrix.)
Denote by $\what M_{j,k}(z)$ the $(j,k)$ entry of the RH matrix $\what
M(z)=\what M(z;r_1,r_2,s_1,s_2,\tau)$. Then the entries in the top left
$2\times 2$ block of $\what M(z)$ can be expressed by the formulas
\begin{equation}\label{M11:Airy:op} \begin{array}{l}
\displaystyle\what M_{1,1}(z) = \int_{0}^{\infty}(\Een-\Kaa_{\Ai,\si})^{-1}(x,0)\, \wtil b_z(x)\ud x,\\
\displaystyle\what M_{2,1}(z)= -\ct\int_{0}^{\infty}\!\!\!\int_{0}^{\infty}
(\Een-\Kaa_{\Ai,\si})^{-1}(x,0)\,\Ai(x+y+\si)\wtil b_z(y) \ud x\ud y,\\
\displaystyle\what M_{1,2}(z) =
-\ct^{-1}\int_{0}^{\infty}\!\!\!\int_{0}^{\infty}
(\Een-\Kaa_{\Ai,\si})^{-1}(x,0)\,\Ai(x+y+\si)b_z(y) \ud x\ud y,\\
\displaystyle\what M_{2,2}(z) =
\int_0^{\infty}(\Een-\Kaa_{\Ai,\si})^{-1}(x,0)\, b_z(x) \ud x,
\end{array}\end{equation}
where we use the notations \eqref{Dirac:delta} and \eqref{C:RH}--\eqref{bzx:FV}
(or \eqref{sigma:symm} in the symmetric case). The entries in the bottom left
$2\times 2$ block of $\what M(z)$ can be obtained by combining the above
expressions with equations \eqref{diffeq:pee:3}--\eqref{diffeq:pee:4} below.
\end{theorem}

Theorem~\ref{theorem:Airyformulas:FV} is proved in
Section~\ref{section:prooftheoremAiryRH}. The proof makes heavy use of the
Tracy-Widom functions defined in Section~\ref{subsection:TW}.

\begin{corollary}\label{cor:AiryRHpee} The first two entries of the vector $\pee(z)$ in
\eqref{pee:firsttwocolumns} are given by
\begin{equation}\begin{array}{l}
\label{p1:Airy:op} \displaystyle p_1(z) =  \int_0^{\infty}(\Een-\Kaa_{\Ai,\si})^{-1}(x,0)\wtil\Aa_{z}(x)\ud x, \\
\displaystyle p_2(z) =
\int_0^{\infty}(\Een-\Kaa_{\Ai,\si})^{-1}(x,0)\Aa_{z}(x)\ud x,
\end{array}\end{equation}
where
\begin{equation}\label{curlyA:FV} \begin{array}{l} \displaystyle\Aa_{z}(x) = b_z(x)
-\ct\int_{0}^{\infty} \Ai(x+y+\si)\wtil b_{z}(y) \ud y,
\\ \displaystyle\wtil\Aa_{z}(x) =
\wtil b_{z}(x)- \ct^{-1}\int_{0}^{\infty} \Ai(x+y+\si)b_{z}(y)\ud y.
\end{array}\end{equation}
\end{corollary}


\subsection{Differential equations for the columns of $M(z)$}

In the proof of Theorem~\ref{theorem:Airyformulas:FV}, we will use the
differential equations for the columns of the RH matrix $\what M(z)$ (or
$M(z)$). Interestingly, the coefficients in these differential equations
contain the Hastings-McLeod solution $q(x)$ to Painlev\'e~II. We also need the
associated \emph{Hamiltonian}
\begin{equation} \label{def:Hamiltonian}
    u(x) := (q'(x))^2-xq^2(x)-q^4(x).
\end{equation}

\begin{proposition}\label{prop:diffeq:pee}
(System of differential equations.)

(a) Let the vector $\mm(z)=\mm(z;r_1,r_2,s_1,s_2,\tau)$ be one of the columns
of $\what M(z)$, or a fixed linear combination of them, and denote its entries
by $m_j(z)$, $j=1,\ldots,4$. Then with the prime denoting the derivative with
respect to $z$, we have
\begin{align}
\nonumber r_1^{-2} m_1'' &= 2\tau m_1'+C^2\ct^{-1} q(\si)m_2'
\\ \label{diffeq:pee:1} &
\qquad\qquad + \left[C q^2(\si)- z+2s_1/r_1 -r_1^2\tau^2\right] m_1
-\left[C\ct^{-1} q'(\si)\right]m_2, \\
\nonumber r_2^{-2} m_2'' &= -C^2\ct q(\si) m_1'- 2\tau m_2'
\\ \label{diffeq:pee:2} &
\qquad\qquad +\left[C q^2(\si)+z+2s_2/r_2-r_2^2\tau^2\right] m_2
-\left[C\ct q'(\si)\right]m_1,\\
\label{diffeq:pee:3} r_1im_3 &= m_1'-(C^{-1} u(\si)-s_1^2+r_1^2\tau) m_1-C^{-1}\ct^{-1}q(\si) m_2, \\
\label{diffeq:pee:4} r_2im_4 &= m_2'+C^{-1}\ct q(\si)
m_1+(C^{-1}u(\si)-s_2^2+r_2^2\tau) m_2,
\end{align}
where the constants $C,\ct,\si$ are defined in \eqref{C:RH} and we denote by
$q,u$ the Hastings-McLeod function and the associated Hamiltonian.

(b) Conversely, any vector $\mm(z)$ that solves
\eqref{diffeq:pee:1}--\eqref{diffeq:pee:4} is a fixed (independent of $z$)
linear combination of the columns of $\what M(z)$.
\end{proposition}

Proposition~\ref{prop:diffeq:pee} is proved in
Section~\ref{subsection:proof:diffeqpee} with the help of Lax pair
calculations. There are similar differential equations with respect to the
parameters $s_1$, $s_2$ or $\tau$ but they will not be needed.

\subsection{Outline of the paper}

The remainder of this paper is organized as follows. In
Section~\ref{subsection:proof:tacnodeder} we establish
Theorem~\ref{theorem:tacnodeder} about the derivative of the RH tacnode kernel.
In Section~\ref{section:proofsAFJVV} we prove Proposition~\ref{prop:FVkernel}
and Theorems~\ref{theorem:identify:kernels} and \ref{theorem:tacnodeder:FV}
about the Ferrari-Vet\H o tacnode kernel. In
Section~\ref{section:prooftheoremAiryRH} we prove
Theorem~\ref{theorem:Airyformulas:FV} about the Airy resolvent formulas for the
entries of the RH matrix $\what M(z)$. Finally, in Section~\ref{section:RHLax}
we use Lax pair calculations to prove Proposition~\ref{prop:diffeq:pee}.

\section{Proof of Theorem~\ref{theorem:tacnodeder}}
\label{subsection:proof:tacnodeder}

Throughout the proof we use the parametrization $s_j=\si_j s$ with $\si_j$
fixed, $j=1,2$. We also assume $r_1,r_2>0$ to be fixed. The RH matrix $\what
M(z)=\what M(z;s,\tau)$ satisfies the differential equation
\begin{equation}\label{Lax:s0}
\frac{\pp}{\pp s}\what M(z) =V(z)\what M(z),
\end{equation}
for a certain coefficient matrix $V(z)=V(z;s,\tau)$. This is described in more
detail in Section~\ref{subsection:Lax}. At this moment, we only need to know
that
\begin{equation}\label{V:specific}
V(z)= -2iz\begin{pmatrix} 0 & 0 & 0 & 0 \\
0 & 0 & 0 & 0 \\
\si_1 & 0 & 0 & 0 \\
0 & \si_2 & 0 & 0
\end{pmatrix}+V(0),
\end{equation}
where the matrix $V(0)$ is independent of $z$. We will also need the symmetry relation
\begin{eqnarray}\label{symmetry:inversetranspose:0}
\what M^{-1}(z;s,\tau) &=& K^{-1} \what M^T(z;s,-\tau) K,
\end{eqnarray}
where the superscript ${}^T$ denotes the transpose and
\begin{equation}\label{permmces:K0}
K = \begin{pmatrix} 0 & I_2 \\
-I_2 & 0 \end{pmatrix} \end{equation}
with $I_2$ denoting the identity matrix of
size $2\times 2$. This symmetry relation is a consequence of Lemma~\ref{lemma:symmetries} below.

We are now ready to prove Theorem~\ref{theorem:tacnodeder}. Abbreviating $\what M(z):=\what M(z;s,\tau)$ for the moment, we start by calculating
\begin{eqnarray*}
\frac{\pp}{\pp s}\left[\what M^{-1}(v) \what M(u)\right]
&=& \what M^{-1}(v)\left(\frac{\pp}{\pp s}\left[\what M(u)\right]\what M^{-1}(u)-\frac{\pp}{\pp s}\left[\what M(v)\right]\what M^{-1}(v)\right)\what M(u) \\
&=& \what M^{-1}(v)(V(u)-V(v) ) \what M(u) \\
&=& -2i(u-v)\what M^{-1}(v)\begin{pmatrix} 0 & 0\\
\Sigma & 0
\end{pmatrix}\what M(u),
\end{eqnarray*}
with $\Sigma:=\diag(\si_1,\si_2)$, where the last equality follows by \eqref{V:specific}. With the help of the above calculation we now obtain
\begin{eqnarray*}  \frac{\pp}{\pp s} K\tac(u,v;s,\tau) &=&
\frac{\pp}{\pp s} \frac{1}{2\pi i(u-v)}
\begin{pmatrix} 0&0&1&1\end{pmatrix}
\what M^{-1}(v;s,\tau) \what M(u;s,\tau)
\begin{pmatrix} 1&1&0&0\end{pmatrix}^T
\\
&=& -\frac{1}{\pi}\begin{pmatrix} 0&0&1&1\end{pmatrix}\what M^{-1}(v;s,\tau)\begin{pmatrix} 0 & 0\\
\Sigma & 0
\end{pmatrix}\what M(u;s,\tau)\begin{pmatrix} 1&1&0&0\end{pmatrix}^T\\
&=& -\frac{1}{\pi}\begin{pmatrix} 1&1&0&0\end{pmatrix}\what M^{T}(v;s,-\tau)\begin{pmatrix} \Sigma & 0\\
0 & 0
\end{pmatrix}\what M(u;s,\tau)\begin{pmatrix} 1&1&0&0\end{pmatrix}^T\\
&=& -\frac{1}{\pi}\pee^T(v;s,-\tau) \begin{pmatrix}
\Sigma & 0 \\
0&0
\end{pmatrix}\pee(u;s,\tau),
\end{eqnarray*}
where the third equality follows from
\eqref{symmetry:inversetranspose:0} and the fourth one from \eqref{pee:firsttwocolumns}. This proves
\eqref{rank2:tacnode:RH}. By integrating this equality, we obtain
\eqref{integrated:kernel}, due to the fact that the entries of $\pee(z;s,\tau)$
go to zero for $s\to +\infty$ if $\si_1,\si_2>0$. The latter fact follows from
\cite[Sec.~3]{DKZ} if $\tau=0$ and is established similarly for general
$\tau\in\er$. This ends the proof of Theorem~\ref{theorem:tacnodeder}. $\bol$

\section{Proofs of Proposition~\ref{prop:FVkernel} and Theorems~\ref{theorem:tacnodeder:FV}, \ref{theorem:identify:kernels}}
\label{section:proofsAFJVV}

\subsection{Proof of Proposition~\ref{prop:FVkernel}}
\label{subsection:proof:FVkernel}

Denote by $\vecA_{\si}$ the operator on $L^2([0,\infty))$ that acts on the
function $f\in L^2([0,\infty))$ by the rule
\begin{equation}\label{operator:A1}
\left[\vecA_{\si} f\right](x) = \int_0^{\infty} \Ai(x+y+\si) f(y)\ud y.
\end{equation}
Observe that
\begin{equation}\label{operator:A2} \vecA_{\si}^2 =
\Kaa_{\Ai,\si},
\end{equation}
on account of \eqref{Airy:kernel}.

Now the kernel $\mathcal L\tac(u,v)$ in \cite[Eq.~(5)]{FV} in the single time
case $\tau_1=\tau_2=\tau$ has the form\footnote{The notations $\si$,
$\wtil\si$, $b_{\tau,\si+\xi}^{\lam}(x+\wtil\si)$,
$B_{\tau,\si+\xi}^{\lam}(x+\wtil\si)$,
$b_{\lam^{1/3}\tau,\lam^{2/3}\si-\lam^{1/6}\xi}^{\lam^{-1}}(x+\wtil\si)$,
$B_{\lam^{1/3}\tau,\lam^{2/3}\si-\lam^{1/6}\xi}^{\lam^{-1}}(x+\wtil\si)$ in
\cite{FV} correspond to our notations $\Sigma$, $\si$, $\lam^{1/6}\wtil
b_{-\tau,\xi}(x)$, $\left[\vecA_{\si} b_{-\tau,\xi}\right](x)$,
$\lam^{-1/6}b_{-\tau,\xi}(x)$, $\left[\vecA_{\si} \wtil
b_{-\tau,\xi}\right](x)$, respectively. The notation $K_{\Ai}^{(-\tau,\tau)}(\si+\xi_1,\si+\xi_2)$ in \cite[Eq.~(11)]{FV} corresponds to the first term in the right
hand side of \eqref{FVkernel:long}.}
\begin{eqnarray}\nonumber C^{-1}\mathcal L\tac(u,v;\si,\tau)
&=&
\int_0^{\infty}b_{\tau,u}(x)b_{-\tau,v}(x)\ud x\\
\nonumber &&+\int_0^{\infty}\!\!\!\int_0^{\infty}
(\Een-\Kaa_{\Ai,\si})^{-1}(x,y)\left[\vecA_{\si}
b_{\tau,u}\right](x)\left[\vecA_{\si}
b_{-\tau,v}\right](y)\ud x\ud y\\
\nonumber&& -\lam^{1/6}\int_0^{\infty}\!\!\!\int_0^{\infty}
(\Een-\Kaa_{\Ai,\si})^{-1}(x,y) \left[\vecA_{\si}
b_{\tau,u}\right](x)\wtil b_{-\tau,v}(y)\ud x\ud y\\
\nonumber&& +
\lam^{1/3}\int_0^{\infty}\wtil b_{\tau,u}(x)\wtil b_{-\tau,v}(x)\ud x\\
\nonumber&& +\lam^{1/3}\int_0^{\infty}\!\!\!\int_0^{\infty}
(\Een-\Kaa_{\Ai,\si})^{-1}(x,y)\left[\vecA_{\si} \wtil
b_{\tau,u}\right](x)\left[\vecA_{\si} \wtil
b_{-\tau,v}\right](y)\ud x\ud y\\
\label{FVkernel:long}&&-\lam^{1/6}\int_0^{\infty}\!\!\!\int_0^{\infty}
(\Een-\Kaa_{\Ai,\si})^{-1}(x,y) \left[\vecA_{\si} \wtil
b_{\tau,u}\right](x)b_{-\tau,v}(y)\ud x\ud y,
\end{eqnarray}
where again $C=(1+\lam^{-1/2})^{1/3}$ and \eqref{Dirac:delta}.
The second term in the right hand side of
\eqref{FVkernel:long} equals
\begin{multline}\label{FVkernel:term2}
\int_0^{\infty}\!\!\!\int_0^{\infty} (\Een-\Kaa_{\Ai,\si})^{-1}(x,y)
\left[\vecA_{\si} b_{\tau,u}\right](x)\left[\vecA_{\si}b_{-\tau,v}\right](y)\ud x\ud y \\
= \int_0^{\infty}\!\!\!\int_0^{\infty} (\Een-\Kaa_{\Ai,\si})^{-1}(x,y)\
b_{\tau,u}(x)b_{-\tau,v}(y)\ud x\ud y-\int_0^{\infty}
b_{\tau,u}(x)b_{-\tau,v}(x)\ud x,
\end{multline}
where we used that
$$ \vecA_{\si}(\Een-\Kaa_{\Ai,\si})^{-1}\vecA_{\si} = (\Een-\Kaa_{\Ai,\si})^{-1}\Kaa_{\Ai,\si} = (\Een-\Kaa_{\Ai,\si})^{-1}-\Een, $$
on account of \eqref{operator:A2}. Next, the third term in the right hand
side of \eqref{FVkernel:long} can be written as
\begin{multline}\label{FVkernel:term3}  -\lam^{1/6}\int_0^{\infty}\!\!\!\int_0^{\infty}
(\Een-\Kaa_{\Ai,\si})^{-1}(x,y) \left[\vecA_{\si}
b_{\tau,u}\right](x)\wtil b_{-\tau,v}(y)\ud x\ud y\\
= -\lam^{1/6}\int_0^{\infty}\!\!\!\int_0^{\infty}
(\Een-\Kaa_{\Ai,\si})^{-1}(x,y)\ b_{\tau,u}(x)\left[\vecA_{\si}\wtil
b_{-\tau,v}\right](y)\ud x\ud y,
\end{multline}
since the operators $(\Een-\Kaa_{\Ai,\si})^{-1}$ and $\vecA_{\si}$ commute.
Inserting \eqref{FVkernel:term2}--\eqref{FVkernel:term3} for the second and third
term in the right hand side of \eqref{FVkernel:long}, we obtain
\begin{eqnarray}\nonumber & & C^{-1}\mathcal L\tac(u,v;\si,\tau)
\\ \nonumber &=&
\lam^{1/3}\int_0^{\infty}\wtil b_{\tau,u}(x)\wtil b_{-\tau,v}(x)\ud x\\
\nonumber && +\int_0^{\infty}\!\!\!\int_0^{\infty}
(\Een-\Kaa_{\Ai,\si})^{-1}(x,y)\ b_{\tau,u}(x)b_{-\tau,v}(y)\ud x\ud y\\
\nonumber && -\lam^{1/6}\int_0^{\infty}\!\!\!\int_0^{\infty}
(\Een-\Kaa_{\Ai,\si})^{-1}(x,y)\ b_{\tau,u}(x)\left[\vecA_{\si}\wtil b_{-\tau,v}\right](y)\ud x\ud y\\
\nonumber&& +\lam^{1/3}\int_0^{\infty}\!\!\!\int_0^{\infty}
(\Een-\Kaa_{\Ai,\si})^{-1}(x,y)\left[\vecA_{\si} \wtil
b_{\tau,u}\right](x)\left[\vecA_{\si} \wtil
b_{-\tau,v}\right](y)\ud x\ud y\\
\label{FVkernel:longbis}&&-\lam^{1/6}\int_0^{\infty}\!\!\!\int_0^{\infty}
(\Een-\Kaa_{\Ai,\si})^{-1}(x,y) \left[\vecA_{\si} \wtil
b_{\tau,u}\right](x)b_{-\tau,v}(y)\ud x\ud y,
\end{eqnarray}
which is equivalent to the desired result \eqref{tacnode:FV}.
 $\bol$

\subsection{Tracy-Widom functions and their properties}
\label{subsection:TW}

In some of the remaining proofs we need the functions
\begin{equation}\label{Q:Airy:def}
\begin{array}{l}\displaystyle Q(x) = \int_0^{\infty}(\Een-\Kaa_{\Ai,\si})^{-1}(x,y)\Ai(y+\si)\ud y\\
\displaystyle P(x) =
\int_0^{\infty}(\Een-\Kaa_{\Ai,\si})^{-1}(x,y)\Ai'(y+\si)\ud y,
\end{array}\end{equation}
with the usual notation \eqref{Dirac:delta}. Note that $Q,P$ depend on $\si$
but we omit this dependence from the notation. The functions $Q,P$ originate
from the seminal paper of Tracy-Widom \cite{TW1}. Below we list some of their
properties. These properties can all be found in \cite{TW1}, see also
\cite[Sec.~3.8]{AGZ} for a text book treatment. One should take into account
that our function $R(x,y)=R_{\si}(x,y)$ equals the one in \cite{TW1},
\cite[Sec.~3.8]{AGZ} at the shifted arguments $x+\si$, $y+\si$, and similarly
for $Q(x)$ and $P(x)$.

\begin{lemma} The functions $Q,P$ and $R=R_{\si}$ satisfy the differential
equations
\begin{eqnarray}
\label{diff:Rxy}\left(\frac{\pp}{\pp x}+\frac{\pp}{\pp y}\right)R(x,y) &=&
R(x,0)R(0,y)-Q(x)Q(y), \\
\label{diff:Rsigma}\frac{\pp}{\pp \si} R(x,y) &=& -Q(x)Q(y),\\
\label{diff:Qx} Q'(x) &=& P(x)+q R(x,0)- u Q(x), \\
\label{diff:Px}  P'(x) &=& (x+\si-2v)Q(x)+p R(x,0)+uP(x),
\end{eqnarray}
where
\begin{eqnarray}
\label{TW:q} q&=&Q(0)\\
\label{TW:p} p&=&P(0)\\
\label{TW:u} u&=&\int_{0}^{\infty} Q(x)\Ai(x+\si)\ud x\\
\label{TW:v} v&=&\int_{0}^{\infty} Q(x)\Ai'(x+\si)\ud x =\int_{0}^{\infty}
P(x)\Ai(x+\si)\ud x.
\end{eqnarray}
\end{lemma}

Note that $q,p,u,v$ are all functions of $\si$, although we do not show this in
the notation. They satisfy the following differential equations with respect to
$\si$ \cite{TW1}, \cite[Sec.~3.8]{AGZ}
\begin{eqnarray}
\label{TW:rel1} q' &=& p-qu \\
\label{TW:rel2} p' &=& \si q+pu-2qv \\
\label{TW:rel3} u' &=& -q^2\\
\label{TW:rel4} v' &=& -pq.
\end{eqnarray}
It is known that $q=q(\si)$ is the Hastings-McLeod solution to the Painlev\'
e~II equation \eqref{def:Painleve2}--\eqref{def:HastingMcLeod}. Moreover,
$u=u(\si)$ is the Hamiltonian \eqref{def:Hamiltonian}, and
\begin{eqnarray} \label{TW:rel5} 2v &=& u^2-q^2.
\end{eqnarray}

Finally, we establish the following lemma.

\begin{lemma}\label{lemma:RQequiv}
For any function $b$ on $[0,\infty)$ we have
\begin{equation}\label{operator:Airy:1} \int_{0}^{\infty}Q(x)b(x)\ud x = \int_{0}^{\infty}\!\!\!
\int_{0}^{\infty} (\Een-\Kaa_{\Ai,\si})^{-1}(x,0)\,\Ai(x+y+\si) b(y)\ud x\ud y
\end{equation}
and \begin{equation}\label{operator:Airy:2}  \int_0^{\infty}\!\!\!
\int_{0}^{\infty}Q(x)\Ai(x+y+\si)b(y)\ud x\ud y =\int_0^{\infty}R(x,0) b(x)\ud
x,
\end{equation}
with the usual notations $R=R_{\si}$ and  \eqref{Dirac:delta}.
\end{lemma}

\begin{proof}
Recall the operator $\vecA_{\si}$ defined in
\eqref{operator:A1}--\eqref{operator:A2} and let $\delta_0$ denote the Dirac
delta function at $0$. We have
\begin{eqnarray*} \int_{0}^{\infty}Q(x)b(x)\ud x
&=& \int_{0}^{\infty}\!\!\! \int_{0}^{\infty}
(\Een-\Kaa_{\Ai,\si})^{-1}(x,y)\,\left[\vecA_{\si}\delta_0\right](y)\, b(x)\ud
x\ud y \\
&=& \int_{0}^{\infty}\!\!\!\int_{0}^{\infty}
(\Een-\Kaa_{\Ai,\si})^{-1}(x,y)\,\delta_0(y)\, \left[\vecA_{\si} b\right](x)\ud
x\ud y,
\end{eqnarray*}
where we used that the operators $\vecA_{\si}$ and $\Kaa_{\Ai,\si}$ commute.
This proves \eqref{operator:Airy:1}. Next,
\begin{eqnarray*} && \int_0^{\infty}\!\!\!
\int_{0}^{\infty}Q(x)\Ai(x+y+\si)b(y)\ud x\ud y \\ &=& \int_0^{\infty}\!\!\!
\int_{0}^{\infty}
(\Een-\Kaa_{\Ai,\si})^{-1}(x,y)\,\left[\vecA_{\si}\delta_0\right](y)\,
\left[\vecA_{\si}b\right](x)\ud x\ud y \\
&=& \int_{0}^{\infty}\!\!\!\int_0^{\infty}R(x,y)\delta_0(y) b(x)\ud x\ud
y,\end{eqnarray*} where we used that
$$\vecA_{\si}(\Een-\Kaa_{\Ai,\si})^{-1}\vecA_{\si} =
(\Een-\Kaa_{\Ai,\si})^{-1}\Kaa_{\Ai,\si}=\Rop_{\Ai,\si}.$$  This proves
\eqref{operator:Airy:2}.
\end{proof}

\subsection{Proof of Theorem~\ref{theorem:tacnodeder:FV}}
\label{subsection:proof:tacnodeder:FV}

\begin{lemma}\label{lemma:equiv:ways}
The formulas \eqref{p1hat:op} can be written in the equivalent ways
\begin{eqnarray}
\label{p1hat:R} \what p_1(z;\si,\tau)
&=& \wtil \Aa_{\tau,z}(0)+ \int_0^{\infty}
R(x,0)\wtil\Aa_{\tau,z}(x)\ud x \\
\label{p1hat:Q} &=& \wtil
b_{\tau,z}(0)-\lam^{-1/6}\int_0^{\infty} Q(x)\Aa_{\tau,z}(x)\ud x \\
\what p_2(z;\si,\tau) &=&
\label{p2hat:R} \Aa_{\tau,z}(0)+\int_0^{\infty} R(x,0)\Aa_{\tau,z}(x)\ud x \\
\label{p2hat:Q}&=& b_{\tau,z}(0)-\lam^{1/6}\int_0^{\infty}
Q(x)\wtil\Aa_{\tau,z}(x)\ud x,
\end{eqnarray}
with the notations $R=R_{\si}$ and \eqref{Q:Airy:def}.
\end{lemma}

\begin{proof}
Immediate from Lemma~\ref{lemma:RQequiv} and the definitions.
\end{proof}

\begin{lemma}
We have
\begin{eqnarray}\label{intparts:FV:0} (1+\lam^{-1/2})\frac{\pp}{\pp
\si}b_{\tau,z}(x) &=& \lam^{-1/2}\frac{\pp}{\pp x} b_{\tau,z}(x), \\
\label{intparts:FV:00} (1+\lam^{-1/2})\frac{\pp}{\pp \si}\wtil b_{\tau,z}(x)
&=& \frac{\pp}{\pp x} \wtil b_{\tau,z}(x),
\\ \label{intparts:FV} (1+\lam^{-1/2})
\frac{\pp}{\pp\si}\Aa_{\tau,z}(x) &=& \lam^{-1/2}\frac{\pp}{\pp
x}\Aa_{\tau,z}(x) +\lam^{1/6} \Ai(x+\si)\wtil b_{\tau,z}(0).
\end{eqnarray}
\end{lemma}

\begin{proof}
The first two formulas are obvious from the definitions
\eqref{temperature:sigma}--\eqref{C:FV0}. For the last formula, we calculate
\begin{eqnarray*}
& & \left[(1+\lam^{-1/2})\frac{\pp}{\partial \si}-\lam^{-1/2}\frac{\pp}{\pp
x}\right]\Aa_{\tau,z}(x) \\ &=& \left[(1+\lam^{-1/2})\frac{\partial}{\partial
\si}-\lam^{-1/2}\frac{\pp}{\pp x}\right] \left(b_{\tau,z}(x)
-\lam^{1/6}\int_{0}^{\infty} \Ai(x+y+\si)\wtil b_{\tau,z}(y) \ud y\right) \\
&=& -\lam^{1/6}\left[(1+\lam^{-1/2})\frac{\pp}{\pp
\si}-\lam^{-1/2}\frac{\pp}{\pp x}\right] \left(
\int_{0}^{\infty} \Ai(x+y+\si)\wtil b_{\tau,z}(y) \ud y\right) \\
&=& -\lam^{1/6}\int_{0}^{\infty} \left(\Ai'(x+y+\si)\wtil b_{\tau,z}(y) + \Ai(x+y+\si)\wtil b_{\tau,z}'(y) \right) \ud y \\
&=& \lam^{1/6}\Ai(x+\si)\wtil b_{\tau,z}(0),
\end{eqnarray*} where the second and third
equalities use \eqref{intparts:FV:0} and \eqref{intparts:FV:00} respectively,
and the last equality follows from integration by parts. This proves the lemma.
\end{proof}

Now we prove Theorem~\ref{theorem:tacnodeder:FV}. With the help of
\eqref{intparts:FV}, the derivative of \eqref{tacnode:FV} with respect to $\si$
becomes
\begin{eqnarray}
\nonumber && (1+\lam^{-1/2})^{2/3} \frac{\pp}{\pp \si} \mathcal
L\tac(u,v;\si,\tau)
\\ \nonumber &=& (1+\lam^{-1/2})^{2/3}\lam^{1/3} C \frac{\pp}{\pp
\si}\int_0^{\infty}\wtil b_{\tau,u}(x)\wtil b_{-\tau,v}(x)\ud x\\
\nonumber && +
\lam^{-1/2}\int_0^{\infty}\!\!\!\int_0^{\infty}(\Een-\Kaa_{\Ai,\si})^{-1}(x,y)\
\left(\Aa_{\tau,u}'(x)\Aa_{-\tau,v}(y)+\Aa_{\tau,u}(x)\Aa_{-\tau,v}'(y)\right)\ud
x\ud y\\ \nonumber && + \lam^{1/6}\wtil
b_{\tau,u}(0)\int_0^{\infty} Q(y)\Aa_{-\tau,v}(y)\ud y\\
\nonumber && + \lam^{1/6}\wtil b_{-\tau,v}(0)\int_0^{\infty} Q(x)
\Aa_{\tau,u}(x)\ud x\\
\label{intparts:FV:long}&&
+(1+\lam^{-1/2})\int_0^{\infty}\!\!\!\int_0^{\infty}\!\!
\left(\frac{\pp}{\pp\si} R(x,y)\right) \Aa_{\tau,u}(x)\Aa_{-\tau,v}(y)\ud x\ud
y,
\end{eqnarray}
where we used the definition of $Q$ in \eqref{Q:Airy:def}. The first term in
the right hand side of \eqref{intparts:FV:long} can be written as
\begin{equation}\label{intparts:FV:2}
(1+\lam^{-1/2})^{2/3}\lam^{1/3} C \frac{\pp}{\pp
\si}\int_0^{\infty}\wtil b_{\tau,u}(x)\wtil b_{-\tau,v}(x)\ud x =  -\lam^{1/3}\wtil b_{\tau,u}(0)\wtil b_{-\tau,v}(0),
\end{equation}
on account of \eqref{temperature:sigma} and \eqref{b:FV}. The second term in
the right hand side of \eqref{intparts:FV:long} can be written as
\begin{multline}\label{intparts:FV:3}
\lam^{-1/2}\int_0^{\infty}\!\!\!\int_0^{\infty}(\Een-\Kaa_{\Ai,\si})^{-1}(x,y)\
\left(\Aa_{\tau,u}'(x)\Aa_{-\tau,v}(y)+\Aa_{\tau,u}(x)\Aa_{-\tau,v}'(y)\right)\ud
x\ud y \\ =
-\lam^{-1/2}\left(\int_0^{\infty}\!\!\!\int_0^{\infty}\!\!\left[\left(\frac{\pp}{\pp
x}+\frac{\pp}{\pp y}\right)R(x,y)\right] \Aa_{\tau,u}(x)\Aa_{-\tau,v}(y)\ud
x\ud y
 +\Aa_{\tau,u}(0)\Aa_{-\tau,v}(0)\right.\\
\left. +\Aa_{\tau,u}(0)\int_0^{\infty}R(0,y)\Aa_{-\tau,v}(y)\ud y
+\Aa_{-\tau,v}(0)\int_0^{\infty}R(x,0)\Aa_{\tau,u}(x)\ud x\right),
\end{multline}
where we used \eqref{Dirac:delta} and integration by parts. Finally, we observe
that
\begin{multline}\label{intparts:FV:4}
\left[ (1+\lam^{-1/2})\frac{\pp}{\pp\si}-\lam^{-1/2}\left(\frac{\pp}{\pp
x}+\frac{\pp}{\pp y}\right)\right]R(x,y) = -\lam^{-1/2} R(x,0)R(0,y)-Q(x)Q(y),
\end{multline}
on account of \eqref{diff:Rxy}--\eqref{diff:Rsigma}. Inserting
\eqref{intparts:FV:2}--\eqref{intparts:FV:4} in the right hand side of
\eqref{intparts:FV:long}, we obtain
\begin{eqnarray}
\nonumber && (1+\lam^{-1/2})^{2/3} \frac{\pp}{\pp \si} \mathcal
L\tac(u,v;\si,\tau) \\ \nonumber &=& -\lam^{1/3}\left(\wtil
b_{\tau,u}(0)-\lam^{-1/6}\int_0^{\infty}Q(x)\Aa_{\tau,u}(x)\ud
x\right)\left(\wtil
b_{-\tau,v}(0)-\lam^{-1/6}\int_0^{\infty}Q(x)\Aa_{-\tau,v}(x)\ud x\right)
\\ \nonumber && -\lam^{-1/2}\left(\Aa_{\tau,u}(0)+\int_0^{\infty}R(x,0)\Aa_{\tau,u}(x)\ud
x\right) \label{intparts:FV:Qs}
\left(\Aa_{-\tau,v}(0)+\int_0^{\infty}R(0,x)\Aa_{-\tau,v}(x)\ud x\right)\\
&=& -\lam^{1/3}\what p_1(u;\si,\tau)\what p_1(v;\si,-\tau)-\lam^{-1/2}\what
p_2(u;\si,\tau)\what p_2(v;\si,-\tau),
\end{eqnarray}
on account of \eqref{p1hat:Q}--\eqref{p2hat:R}. This proves
\eqref{rank2:tacnode:FV}.

Finally we prove \eqref{integrated:tacnode:FV}. From \cite[Sec.~3.8]{AGZ} we
have the estimate
\begin{equation}\label{R:estimate}
|R_{\si}(x,y)|< C_0 e^{-x-y-2\si}
\end{equation}
for all $x,y,\si>0$, with $C_0>0$ a certain constant. We also have the
asymptotics for the Airy function
\begin{equation}\label{Airy:asy}
\Ai(x)\sim \exp\left(-\frac 23 x^{3/2}\right)/(2\sqrt{\pi}x^{1/4}),\qquad x\to
+\infty.
\end{equation}
Consequently the kernel $\mathcal L\tac(u,v;\si,\tau)$ in \eqref{tacnode:FV},
\eqref{Dirac:delta} goes to zero (at a very fast rate) for $\si\to +\infty$.
Integration of \eqref{rank2:tacnode:FV} then yields
\eqref{integrated:tacnode:FV}. This proves Theorem~\ref{theorem:tacnodeder:FV}.
$\bol$

\subsection{Proof of Theorem~\ref{theorem:identify:kernels}}
\label{subsection:proof:identify}

Let the parameters $r_1,r_2,s_1,s_2$ be given by \eqref{identify:kernels}. As
already observed, the expressions for $C,\si$ in \eqref{C:RH} and
\eqref{C:FV0}, \eqref{temperature:sigma} are equal under this identification.
Similarly, the expressions for $p_1,p_2$ in \eqref{p1:Airy:op} and $\what
p_1,\what p_2$ in \eqref{p1hat:op} are related by
\begin{equation}\label{pphat}
p_j(z) = \sqrt{2\pi} r_j^{1/6} \exp\left(r_j^4 \tau(\Sigma+\frac 23
\tau^2)\right)\what p_j(z),\qquad j=1,2.
\end{equation}
Note that the exponential factor in \eqref{pphat} cancels out in
\eqref{rank2:tacnode:RH}.

Now observe that the formulas for $s_1,s_2$ in \eqref{identify:kernels} are of
the form $s_j=\si_j s$ with
\begin{equation}\label{identify:s12}s:=(\Sigma+\tau^2)/2,\qquad
\si_1:=\lam^{3/4},\qquad \si_2:=1.\end{equation} By \eqref{temperature:sigma}
we then have \begin{equation}  \frac{\pp}{\pp s}\mathcal L\tac(u,v;\si,\tau) =
2\sqrt{\lam}(1+\lam^{-1/2})^{2/3}\frac{\pp}{\pp \si}\mathcal
L\tac(u,v;\si,\tau),
\end{equation} where we consider $\tau$ to be fixed. Theorem~\ref{theorem:identify:kernels} follows from
Theorems~\ref{theorem:tacnodeder} and \ref{theorem:tacnodeder:FV} and the above
observations. $\bol$

\section{Proof of Theorem~\ref{theorem:Airyformulas:FV}}
\label{section:prooftheoremAiryRH}

\subsection{Preparations}
\label{subsection:prep}

\begin{lemma}
The functions $b_z(x)$, $\wtil b_z(x)$ in \eqref{bzx:FV} satisfy the
differential equations
\begin{equation}\label{bzx:FV:diffeq:0}
b_z'(x) = C\frac{\pp}{\pp z} b_z(x),\qquad \wtil b_z'(x) = -C\frac{\pp}{\pp z}
\wtil b_z(x),
\end{equation}
and
\begin{equation}\label{bzx:FV:diffeq:1} r_2^{-2}\frac{\pp^2}{\pp z^2} b_{z}(x)+2\tau \frac{\pp}{\pp z} b_z(x) =
\left(z+Cx+2\frac{s_2}{r_2}-r_2^{2}\tau^2 \right)b_z(x),
\end{equation}
\begin{equation}\label{bzx:FV:diffeq:2} r_1^{-2}\frac{\pp^2}{\pp z^2} \wtil b_{z}(x)-2\tau \frac{\pp}{\pp z} \wtil b_z(x) =
\left(-z+Cx+2\frac{s_1}{r_1}-r_1^{2}\tau^2 \right)\wtil b_z(x).
\end{equation}
\end{lemma}
\begin{proof}
Equation \eqref{bzx:FV:diffeq:0} is obvious. The other equations follow from a
straightforward calculation using the Airy differential equation
$\Ai''(x)=x\Ai(x)$.
\end{proof}

\begin{lemma} The function $\Aa_z(x)$ in \eqref{curlyA:FV} satisfies the
differential equations
\begin{equation}\label{curlyA:FV:diff:1} \frac{\pp}{\pp z} \Aa_{z}(x) =
C^{-1}\left(\Aa_{z}'(x) - \ct\Ai(x+\si)\wtil b_{z}(0)\right)
\end{equation}
and
\begin{multline}\label{curlyA:FV:diff:2} r_2^{-2}\frac{\pp^2}{\pp z^2} \Aa_{z}(x)+2\tau \frac{\pp}{\pp z} \Aa_z(x) =
\left(z+Cx+2\frac{s_2}{r_2}-r_2^{2}\tau^2 \right)\Aa_{z}(x)  \\
+ C \ct\left(\Ai(x+\si)\wtil b_{z}'(0)  - \Ai'(x+\si)\wtil b_{z}(0)\right).
\end{multline}
\end{lemma}

\begin{proof} Equation~\eqref{curlyA:FV:diff:1} follows from the definition of $\Aa_z$, \eqref{bzx:FV:diffeq:0} and
integration by parts. Now we check the formula \eqref{curlyA:FV:diff:2}. From
\eqref{curlyA:FV} we have
\begin{eqnarray*}
&& r_2^{-2}\frac{\pp^2}{\pp z^2}\Aa_{z}(x) \\ &=& r_2^{-2}\frac{\pp^2}{\pp
z^2}b_{z}(x)
-r_2^{-2}\ct\int_0^{\infty}  \Ai(x+y+\si)\frac{\pp^2}{\pp z^2} \wtil b_{z}(y)\ud y \\
&=& r_2^{-2}\frac{\pp^2}{\pp z^2} b_{z}(x) + r_1^{-2}\ct\int_0^{\infty}
\Ai(x+y+\si)\frac{\pp^2}{\pp z^2} \wtil b_{z}(y)\ud y - C\ct\int_0^{\infty}
\Ai(x+y+\si)\wtil b_{z}''(y)\ud y
\end{eqnarray*}
where in the second equality we used $r_1^{-2}+r_2^{-2}=C^3$ and
\eqref{bzx:FV:diffeq:0}. Hence
\begin{multline}\label{proof:curlyA:FV:2}
r_2^{-2}\frac{\pp^2}{\pp z^2}\Aa_{z}(x)+2\tau \frac{\pp}{\pp z}\Aa_{z}(x)=
\left(r_2^{-2}\frac{\pp^2}{\pp z^2} b_{z}(x)+2\tau \frac{\pp}{\pp
z}b_{z}(x)\right) \\ + \ct\int_0^{\infty}
\Ai(x+y+\si)\left(r_1^{-2}\frac{\pp^2}{\pp z^2} \wtil b_{z}(y)-2\tau
\frac{\pp}{\pp z} \wtil b_{z}(y)\right)\ud y -C\ct\int_0^{\infty}
\Ai(x+y+\si)\wtil b_{z}''(y)\ud y.
\end{multline}
In the first two terms in the right hand side of \eqref{proof:curlyA:FV:2} we
use the differential equations
\eqref{bzx:FV:diffeq:1}--\eqref{bzx:FV:diffeq:2}, and in the third term we
integrate by parts twice and subsequently use the Airy differential equation
$\Ai''(x)=x\Ai(x)$. The lemma then follows from a straightforward calculation,
taking into account that
\begin{equation}\label{sigma:FV:simplify}
\left(-z+Cy+2\frac{s_1}{r_1}-r_1^2\tau^2 \right)-C\left(x+y+\sigma\right) =
-z-Cx-2\frac{s_2}{r_2}+r_2^2\tau^2
\end{equation}
thanks to \eqref{C:RH}.
\end{proof}

\begin{lemma}\label{lemma:equiv:ways:bis}
The formulas \eqref{p1:Airy:op} can be written in the equivalent ways
\begin{align} \label{p1:Airy:R:FV} p_1(z)  & =\wtil\Aa_{z}(0)+
\int_{0}^{\infty}R(x,0)\wtil\Aa_{z}(x)\ud x \\
&= \label{p1:Airy:Q:FV} \wtil b_{z}(0)-\ct^{-1}\int_{0}^{\infty}
Q(x)\Aa_z(x)\ud x, \\ \label{p2:Airy:R:FV} p_2(z) &=\Aa_{z}(0)+
\int_{0}^{\infty}R(x,0)\Aa_{z}(x)\ud
x\\
\label{p2:Airy:Q:FV} &= b_z(0)-\ct\int_{0}^{\infty} Q(x)\wtil\Aa_{z}(x)\ud x.
\end{align}
\end{lemma}

\begin{proof} Immediate from Lemma~\ref{lemma:RQequiv} and the definitions.
\end{proof}

\subsection{Differential equation for $p_j$}
\label{subsection:diffeq:pee:check}

In this section we check that the expressions for $p_1(z)$, $p_2(z)$ in
\eqref{p1:Airy:Q:FV}--\eqref{p2:Airy:R:FV} satisfy the differential equation
\eqref{diffeq:pee:1} (with $m_j:=p_j$). From \eqref{p1:Airy:Q:FV} we have
\begin{equation}\label{diff:p1:FV} p_1'(z) =
\frac{\pp}{\pp z}\left[ \wtil
b_{z}(0)\right]-\ct^{-1}\int_{0}^{\infty}Q(x)\frac{\pp}{\pp
z}\left[\Aa_z(x)\right]\ud x.
\end{equation}
We calculate the second term
\begin{eqnarray}\nonumber & &
\int_{0}^{\infty}Q(x)\frac{\pp}{\pp z}\left[\Aa_z(x)\right]\ud x \\
\nonumber &=& C^{-1}\left(\int_{0}^{\infty}Q(x)\Aa_z'(x)\ud x -\ct u \wtil b_{z}(0) \right) \\
\nonumber &=& -C^{-1}\left(\int_{0}^{\infty}
Q'(x)\Aa_z(x)\ud x+q\Aa_z(0)+\ct u\wtil b_{z}(0)\right) \\
\nonumber &=& -C^{-1}\left(\int_{0}^{\infty}
\left[P(x)+qR(x,0)-uQ(x)\right]\Aa_z(x)\ud x+q\Aa_z(0)+\ct u\wtil b_{z}(0)\right) \\
&=& \label{terms3:v:0:FV} -C^{-1}\left(\ct u p_1(z)+q p_2(z)+\int_{0}^{\infty}
P(x)\Aa_z(x)\ud x\right)
\end{eqnarray}
where the first equality uses \eqref{curlyA:FV:diff:1} and \eqref{TW:u}, the
second one uses integration by parts and \eqref{TW:q}, the third one uses
\eqref{diff:Qx}, and the fourth equality uses
\eqref{p1:Airy:Q:FV}--\eqref{p2:Airy:R:FV}.

From \eqref{diff:p1:FV}--\eqref{terms3:v:0:FV} and $r_1^{-2}+r_2^{-2}=C^3$, we
get
\begin{multline}\label{diff:p1:temp:FV} r_1^{-2} p_1'(z) =
r_1^{-2}\frac{\pp}{\pp z}[\wtil b_{z}(0)] +
r_2^{-2}\ct^{-1}\int_{0}^{\infty}Q(x)\frac{\pp}{\pp z}\left[\Aa_z(x)\right]\ud x\\
+C^2\ct^{-1}\left(\ct u p_1(z)+q p_2(z)+\int_{0}^{\infty} P(x)\Aa_z(x)\ud
x\right).
\end{multline}
By differentiating \eqref{diff:p1:temp:FV} with respect to $z$, we get
\begin{multline*}r_1^{-2} p_1''(z)-C^2\ct^{-1}q p_2'(z)-2\tau p_1'(z)
= r_1^{-2}\frac{\pp^2}{\pp z^2}\left[  \wtil b_{z}(0)\right]-2\tau
\frac{\pp}{\pp z}\left[  \wtil b_{z}(0)\right] \\ + \ct^{-1}\int_{0}^{\infty}
Q(x)\left(r_2^{-2}\frac{\pp^2}{\pp z^2}\left[\Aa_z(x)\right]+2\tau
\frac{\pp}{\pp z}\left[ \Aa_z(x)\right]\right)\ud x \\ C^2\ct^{-1} \left(Du
p_1'(z)+\int_{0}^{\infty} P(x)\frac{\pp}{\pp z}\left[\Aa_z(x)\right]\ud
x\right),
\end{multline*}
where the last term in the left hand side was expanded using
\eqref{diff:p1:FV}. Equivalently,
\begin{multline}\label{terms3:v:FV} r_1^{-2} p_1''(z)-C^2\ct^{-1}q p_2'(z)-2\tau p_1'(z)
= r_1^{-2}\frac{\pp^2}{\pp z^2}\left[  \wtil b_{z}(0)\right]-2\tau
\frac{\pp}{\pp z}\left[  \wtil b_{z}(0)\right]\\ +
C\ct^{-1}\left(C^{-1}\int_{0}^{\infty} Q(x)\left(r_2^{-2}\frac{\pp^2}{\pp
z^2}\left[\Aa_z(x)\right]+2\tau \frac{\pp}{\pp z}\left[
\Aa_z(x)\right]\right)\ud x\right. \\ \left. +C\ct u p_1'(z)+C\int_{0}^{\infty}
P(x)\frac{\pp}{\pp z}\left[\Aa_z(x)\right]\ud x\right).
\end{multline}
We will calculate each of the terms in the right hand side of
\eqref{terms3:v:FV}. We start with
\begin{eqnarray}
C\int_{0}^{\infty} P(x)\frac{\pp}{\pp z}\left[\Aa_z(x)\right]\ud x \nonumber
&=& \int_{0}^{\infty} P(x) \Aa_{z}'(x)\ud x -
\ct v \wtil b_{z}(0)  \\
\nonumber &=& -\left(\int_{0}^{\infty} P'(x) \Aa_{z}(x)\ud x +p \Aa_{z}(0) +\ct
v\wtil b_{z}(0) \right)
\end{eqnarray}
where the first equality follows from \eqref{curlyA:FV:diff:1} and
\eqref{TW:v}, and the second equality uses integration by parts and
\eqref{TW:p}. Consequently,
\begin{multline}\label{term2:v:FV}
C\int_{0}^{\infty} P(x)\frac{\pp}{\pp z}\left[\Aa_z(x)\right]\ud x  =
-\int_{0}^{\infty} (x+\si)Q(x) \Aa_{z}(x)\ud x  \\
-u\int_{0}^{\infty} P(x) \Aa_{z}(x)\ud x-p p_2(z) -2\ct v p_1(z)+\ct v \wtil
b_{z}(0)
\end{multline}
by virtue of \eqref{diff:Px} and \eqref{p1:Airy:Q:FV}--\eqref{p2:Airy:R:FV}.

Next we calculate the third term in the right hand side of \eqref{terms3:v:FV},
\begin{multline}\label{term3:v:FV}
C^{-1}\int_{0}^{\infty}Q(x)\left(r_2^{-2}\frac{\pp^2}{\pp
z^2}\left[\Aa_z(x)\right]+2\tau \frac{\pp}{\pp z}\left[
\Aa_z(x)\right]\right)\ud x  \\ =
C^{-1}\int_{0}^{\infty}\left(z+Cx+2\frac{s_2}{r_2}-r_2^{2}\tau^2 \right) Q(x)
\Aa_{z}(x)\ud x  + \ct(u \wtil b_{z}'(0) - v \wtil b_{z}(0)),
\end{multline}
on account of \eqref{curlyA:FV:diff:2} and \eqref{TW:u}--\eqref{TW:v}. By
adding \eqref{term2:v:FV}--\eqref{term3:v:FV} and canceling terms we get
\begin{multline}\label{term23:v:FV}
 C\int_{0}^{\infty}
P(x)\frac{\pp}{\pp z}\left[\Aa_z(x)\right]\ud x
+C^{-1}\int_{0}^{\infty}Q(x)\left(r_2^{-2}\frac{\pp^2}{\pp
z^2}\left[\Aa_z(x)\right]+2\tau \frac{\pp}{\pp z}\left[
\Aa_z(x)\right]\right)\ud x \\
 =
-u\int_{0}^{\infty} P(x) \Aa_{z}(x)\ud x + C^{-1}\left(z-
2\frac{s_1}{r_1}+r_1^2\tau^2 \right)\int_{0}^{\infty}Q(x) \Aa_{z}(x)\ud x\\
+ \ct u \wtil b_{z}'(0)-2\ct v p_1(z)-p p_2(z),
\end{multline}
where the factor between brackets in front of the second integral in the right
hand side was obtained via \eqref{sigma:FV:simplify}. Finally, we calculate the
fourth term in the right hand side of \eqref{terms3:v:FV},
\begin{eqnarray}\nonumber
C\ct u p_1'(z) &=&  C\ct u \frac{\pp}{\pp z}[\wtil b_{z}(0)]+\ct u^2 p_1(z)+uq
p_2(z)+u\int_{0}^{\infty} P(x)\Aa_z(x)\ud x \\
\label{term4:v:FV} &=&  -\ct u \wtil b_{z}'(0)+\ct u^2 p_1(z)+uq
p_2(z)+u\int_{0}^{\infty} P(x)\Aa_z(x)\ud x
\end{eqnarray}
where the first equality follows from \eqref{diff:p1:FV}--\eqref{terms3:v:0:FV}
and the second one from \eqref{bzx:FV:diffeq:0}. Inserting
\eqref{term23:v:FV}--\eqref{term4:v:FV} in the right hand side of
\eqref{terms3:v:FV} and canceling terms we get
\begin{eqnarray*} & & r_1^{-2}p_1''(z)-C^2\ct^{-1}q p_2'(z)-2\tau p_1'(z) \\ &=& \left(r_1^{-2}\frac{\pp^2}{\pp
z^2}\left[\wtil b_{z}(0)\right]-2\tau \frac{\pp}{\pp z}\left[\wtil
b_{z}(0)\right]\right) +\ct^{-1}\left(z- 2\frac{s_1}{r_1}+
r_1^2\tau^2\right)\int_{0}^{\infty} Q(x)\Aa_{z}(x)\ud x
\\ & & +C [u^2-2v] p_1(z)+C\ct^{-1}[uq-p] p_2(z) \\
&=& \left[-z+ 2\frac{s_1}{r_1}-r_1^2\tau^2 +C (u^2- 2v)\right]
p_1(z)+C\ct^{-1}[uq-p] p_2(z)
\\ &= & \left[-z+ 2\frac{s_1}{r_1}-r_1^2\tau^2 +C q^2\right] p_1(z)-\left[C\ct^{-1} q'\right] p_2(z)
\end{eqnarray*}
where the second equality follows from \eqref{p1:Airy:Q:FV} and
\eqref{bzx:FV:diffeq:2}, and the last equality uses \eqref{TW:rel1} and
\eqref{TW:rel5}. We have established the desired differential equation
\eqref{diffeq:pee:1}.

\subsection{Other differential equations}

Denote by $\what N_{j,k}(z)$ the right hand side of the formula for $\what
M_{j,k}(z)$ in \eqref{M11:Airy:op}. Let the vectors $\pee(z),\mm(z)$ have
entries
$$ p_j(z) =\what N_{j,1}(z)+\what N_{j,2}(z),\qquad m_j(z)=\what N_{j,1}(z)-\what N_{j,2}(z),$$
for $j=1,\ldots,4$. For $\pee(z)$ this is compatible with \eqref{p1:Airy:op}.
We have already shown in Section~\ref{subsection:diffeq:pee:check} that
$\pee(z)$ satisfies the differential equation \eqref{diffeq:pee:1} (with
$m_j:=p_j$). We claim that the same statement holds for $\mm(z)$. This can be
shown by going through the proofs in
Sections~\ref{subsection:prep}--\ref{subsection:diffeq:pee:check} again and
replacing the appropriate plus signs by minus signs and vice versa. We leave
this to the reader. Summarizing, both $\pee(z)$ and $\mm(z)$ satisfy the
differential equation \eqref{diffeq:pee:1}. By symmetry they also satisfy the
differential equation \eqref{diffeq:pee:2}. Finally, \eqref{diffeq:pee:3}--\eqref{diffeq:pee:4} is
valid by construction. Proposition~\ref{prop:diffeq:pee}(b) implies that
$\pee(z)$ and $\mm(z)$ are fixed linear combinations of the columns of $\what
M(z)$. By linearity, the same holds for the $\what N_{j,k}(z)$.

\subsection{Asymptotics}

In view of the previous section, Theorem~\ref{theorem:Airyformulas:FV} will be
proved if we can show that the expressions for $\what M_{j,1}(z)$ and $\what
M_{j,2}(z)$ in \eqref{M11:Airy:op} and \eqref{diffeq:pee:3}--\eqref{diffeq:pee:4} satisfy the
asymptotics for $z\to\infty$ in the RH problem~\ref{rhp:modelM}. It will be
enough to prove the asymptotics for the second column  $\what M_{j,2}(z)$.
Moreover, it is sufficient to let $z$ go to infinity along the positive real
line. Indeed, first observe that the second columns of $M(z)$ and $\what M(z)$
are equal if $\Re z>0$. Furthermore, the second column of $M(z)$ is recessive
with respect to the other columns as $z\to +\infty$, due to
\eqref{M:asymptotics}. So if we can prove that the expressions for $\what
M_{j,2}(z)$ in \eqref{M11:Airy:op} share this same recessive asymptotic
behavior for $z\to +\infty$, then we are done.

The proof of the above asymptotics is now an easy consequence of
\eqref{R:estimate}--\eqref{Airy:asy}. This ends the proof of
Theorem~\ref{theorem:Airyformulas:FV}.

\section{Proof of Proposition~\ref{prop:diffeq:pee}}
\label{section:RHLax}

\subsection{Painlev\'e formulas for the residue matrix $M_1$}

Let the parameters $r_1,r_2>0$ and $s_1,s_2,\tau\in\er$ be fixed. In the proof
of Proposition~\ref{prop:diffeq:pee} we will need Painlev\'e formulas for the
entries of the \lq residue\rq\ matrix $M_1=M_1(r_1,r_2,s_1,s_2,\tau)$ in
\eqref{M:asymptotics}. Write this matrix in entrywise form as
\begin{equation} \label{M1}
    M_1 =: \begin{pmatrix} a & b & ic & id \\ - \wtil b & - \wtil a & i\wtil d & i \wtil c \\
    i e & i \wtil f & -\alpha & \wtil \beta \\ i f & i\wtil e & -\beta & \wtil
    \alpha
    \end{pmatrix}
    \end{equation}
for certain numbers $a, \wtil a, b, \wtil b,c,\wtil c,\ldots$ that depend on
$r_1, r_2, s_1, s_2,\tau$. We will sometimes write $a(r_1,r_2,s_1,s_2,\tau)$,
$b(r_1,r_2,s_1,s_2,\tau)$, etc., to denote the dependence on the parameters.

In the symmetric case $r_1=r_2$, $s_1=s_2$ we are allowed to drop all the
tildes from \eqref{M1}, while in the case $\tau=0$ we can replace all the Greek
letters in \eqref{M1} by their Roman counterparts and put $\wtil d=d$ and
$\wtil f=f$. These are special instances of the next lemma.

\begin{lemma}\label{lemma:symmetry:M1}
(Symmetry relations.) Let $r_1,r_2>0$ and $s_1,s_2,\tau\in\er$ be fixed. The 16
entries $x=x(r_1,r_2,s_1,s_2,\tau)$ of the matrix \eqref{M1} are all
real-valued and they satisfy the symmetry relations
\begin{equation}
x(r_1,r_2,s_1,s_2,\tau) = \wtil x(r_2,r_1,s_2,s_1,\tau),
\end{equation}
for any $x=a,b,c,d,e,f,\alpha,\beta$, and
\begin{equation}
x(r_1,r_2,s_1,s_2,\tau) = \chi(r_1,r_2,s_1,s_2,-\tau),
\end{equation}
for any $x=a,b,\wtil a,\wtil b,c,\wtil c,d,e,\wtil e,f$, where we write
$\chi=\alpha,\beta,\wtil\alpha,\wtil\beta,c,\wtil c,\wtil d,e,\wtil e,\wtil f$,
respectively.
\end{lemma}

The proof of Lemma~\ref{lemma:symmetry:M1} follows from
Section~\ref{subsection:symmetry:RH} below.

Now we relate the entries of the matrix $M_1$ to the Hastings-McLeod solution
$q(x)$ of Painlev\'e~II and the Hamiltonian $u(x)$ in \eqref{def:Hamiltonian}.
The next theorem was proved for the special case $\tau=0$ in \cite{DKZ} and in
the symmetric setting $r_1=r_2$, $s_1=s_2$ in \cite{Delvaux2012,DG}. In the
general case we have the extra exponential factor $D$ in \eqref{C:RH}.

\begin{theorem}\label{theorem:M1:nonsymm} (Painlev\'e formulas.) Let the parameters
$r_1,r_2>0$ and $s_1,s_2,\tau\in\er$ be fixed. The entries in the top right
$2\times 2$ block of \eqref{M1} are given by \begin{eqnarray}
\label{d:Painleve2:nonsymm} d &=& (r_2C\ct)^{-1} q(\sigma)\\
\label{dtil:Painleve2:nonsymm}\wtil d &=& (r_1C)^{-1}\ct
q(\sigma)\\
\label{c:Hamiltonian:nonsymm}c &=& r_1^{-1}(s_1^2-C^{-1} u\left(\sigma\right))
\\
\label{ctil:Hamiltonian:nonsymm} \wtil c &=&
r_2^{-1}(s_2^2-C^{-1}u\left(\sigma\right))
\end{eqnarray}
where $q$ is the Hastings-McLeod solution to Painlev\' e~II
\eqref{def:Painleve2}--\eqref{def:HastingMcLeod}, $u$ is the Hamiltonian
\eqref{def:Hamiltonian}, and with the constants $C,\ct,\si$ given by
\eqref{C:RH}. Moreover, some of the other entries in \eqref{M1} are given by
\begin{eqnarray}
\label{b:PII} b &=& (\wtil c+\tau r_2) d -(r_2^{2}C^{2}\ct)^{-1} q'(\si) \\
\label{btil:PII} \wtil b &=& (c+\tau r_1) \wtil d -(r_1^{2}C^{2})^{-1}\ct q'(\si)\\
\label{beta:PII} \beta &=& (\wtil c-\tau r_2) \wtil d-(r_1r_2C^{2})^{-1}\ct q'(\si)\\
\label{betatil:PII}\wtil\beta &=& (c-\tau r_1) d-(r_1r_2C^{2}\ct)^{-1}q'(\si)
\end{eqnarray}
and \begin{eqnarray}\label{f:PII} r_1f &=& -\frac{r_2}{r_1^2+r_2^2} \frac{\pp
d}{\pp\tau}+(-r_1 c-r_2 \wtil c+r_1^2\tau +r_2^2\tau) b -r_1 d^2\wtil d+r_2
\wtil c^2 d -2s_2 d\\
\label{ftil:PII}r_2\wtil f &=& -\frac{r_1}{r_1^2+r_2^2} \frac{\pp \wtil
d}{\pp\tau} +(-r_1 c-r_2 \wtil c+r_1^2\tau+r_2^2\tau)\wtil b-r_2\wtil d^2 d+r_1
c^2 \wtil d-2s_1\wtil d.
\end{eqnarray}
\end{theorem}

Theorem~\ref{theorem:M1:nonsymm} is proved in
Section~\ref{subsection:proofPII}.

\subsection{Symmetry relations}
\label{subsection:symmetry:RH}

For further use, we collect some elementary results concerning symmetry.

\begin{lemma}\label{lemma:symmetries}
(Symmetry relations.) For any fixed $r_1, r_2,s_1,s_2,\tau$, the RH matrix $M$
satisfies the symmetry relations
\begin{equation}\label{symmetry:conjugate}
\overline{M(\overline{z};r_1,r_2,s_1,s_2,\tau)} =
\begin{pmatrix} I_2 & 0 \\ 0 & -I_2 \end{pmatrix}
M(z;\overline{r_1},\overline{r_2},\overline{s_1},\overline{s_2},\overline{\tau}) \begin{pmatrix} I_2 & 0 \\
0 & -I_2
\end{pmatrix},
\end{equation}
where the bar denotes complex conjugation,
\begin{equation}\label{symmetry:inversetranspose}
M^{-T}(z;r_1,r_2,s_1,s_2,\tau) = K^{-1} M(z;r_1,r_2,s_1,s_2,-\tau) K,
\end{equation} where the superscript ${}^{-T}$ denotes the inverse transpose,
and finally
\begin{equation}\label{symmetry:minuszeta}
M(-z; r_1, r_2, s_1, s_2,\tau) =
 \begin{pmatrix} J & 0 \\ 0 & -J \end{pmatrix}
    M(z; r_2, r_1, s_2, s_1,\tau)   \begin{pmatrix} J & 0
\\ 0 & -J
\end{pmatrix},
\end{equation}
where we denote \begin{equation}\label{permmces:JK}
J=\begin{pmatrix}0&1\\1&0\end{pmatrix},\qquad K = \begin{pmatrix} 0 & I_2 \\
-I_2 & 0 \end{pmatrix}.
\end{equation}
\end{lemma}

\begin{proof} This follows as in \cite[Sec.~5.1]{DKZ}. One easily checks that the left and right hand sides of
\eqref{symmetry:conjugate} satisfy the same RH problem. Then
\eqref{symmetry:conjugate} follows from the uniqueness of the solution to this
RH problem. The same argument applies to \eqref{symmetry:inversetranspose} and
\eqref{symmetry:minuszeta}.
\end{proof}

\begin{corollary}\label{cor:symmetries}
For any fixed $r_1, r_2,s_1,s_2,\tau$, the residue matrix $M_1$ in
\eqref{M:asymptotics}, \eqref{M1} satisfies the symmetry relations
\begin{eqnarray}\label{symmetry:conjugate:M1}
\overline{M_1(r_1,r_2,s_1,s_2,\tau)} &=&
\begin{pmatrix} I_2 & 0 \\ 0 & -I_2 \end{pmatrix}
M_1(\overline{r_1},\overline{r_2},\overline{s_1},\overline{s_2},\overline{\tau}) \begin{pmatrix} I_2 & 0 \\
0 & -I_2
\end{pmatrix},
\\ \label{symmetry:inversetranspose:M1}
M_1^{T}(r_1,r_2,s_1,s_2,\tau) &=& -K^{-1} M_1(r_1,r_2,s_1,s_2,-\tau) K,
\\ \label{symmetry:minuszeta:M1}
M_1(r_1, r_2, s_1, s_2,\tau) &=&
 -\begin{pmatrix} J & 0 \\ 0 & -J \end{pmatrix}
    M_1(r_2, r_1, s_2, s_1,\tau)   \begin{pmatrix} J & 0
\\ 0 & -J
\end{pmatrix},
\end{eqnarray}
with the notations $J,K$ in \eqref{permmces:JK}.
\end{corollary}

Lemma~\ref{lemma:symmetry:M1} is an immediate consequence of
Corollary~\ref{cor:symmetries}.

\subsection{Lax system}
\label{subsection:Lax}

To the RH matrix $M(z)$ there is associated a Lax system of differential
equations
\begin{equation}\label{Lax:system}
\frac{\pp}{\pp z}M =UM,\qquad \frac{\pp}{\pp s}M =VM,\qquad
\frac{\pp}{\partial \tau}M =WM,
\end{equation}
for certain coefficient matrices $U,V,W$. These matrices were obtained in the
symmetric case $r_1=r_2$, $s_1=s_2$ in \cite[Sec.~5.3]{Delvaux2012}, \cite{DG}
and for $\tau=0$ in \cite[Sec.~5.2]{DKZ}. We will consider the general
nonsymmetric case. To take derivatives with respect to $s_1$ or $s_2$, we again
parameterize $s_j=\si_j s$ with $\si_1,\si_2$ fixed and $s$ variable, as in
\eqref{s12:s}.

\begin{lemma} In the general nonsymmetric setting, with the parametrization $s_1=\si_1 s$, $s_2=\si_2 s$,
the coefficient matrices $U,V,W$ in \eqref{Lax:system} take the form
\begin{align}\label{Lax:zeta:nonsymm}
U&= \begin{pmatrix} -r_1 c+r_1^2\tau & r_2 d & r_1 i & 0 \\
-r_1 \wtil d & r_2 \wtil c-r_2^2\tau & 0 & r_2 i \\
(r_1 c^2-r_2 d\wtil d-2s_1+r_1z)i & -(r_1 b+r_2\wtil\beta) i & r_1 c+r_1^2\tau & r_1 d \\
-(r_1\beta+r_2\wtil b) i & (r_2\wtil c^2-r_1 d\wtil d-2s_2-r_2z)i & -r_2\wtil d
& -r_2\wtil c-r_2^2\tau
\end{pmatrix},
\\ \label{Lax:s:nonsymm}
V&= 2\begin{pmatrix} \si_1 c & \si_2 d & -\si_1 i & 0 \\
\si_1\wtil d & \si_2\wtil c & 0 & \si_2 i \\
\si_1(-c^2+\frac{r_2}{r_1} d\wtil d+\frac{\si_1}{r_1}s-z)i & (\si_1 b-\si_2\wtil\beta)i & -\si_1 c & -\si_1 d \\
(\si_1\beta-\si_2\wtil b)i & \si_2(\wtil c^2-\frac{r_1}{r_2}d\wtil
d-\frac{\si_2}{r_2} s-z)i & -\si_2 \wtil d & -\si_2 \wtil c
\end{pmatrix},
\\ \label{Lax:tau:nonsymm}
W& = (r_1^2+r_2^2)\begin{pmatrix} \frac{r_1^2}{r_1^2+r_2^2}z & -b & 0 & -di \\
-\wtil b & -\frac{r_2^2}{r_1^2+r_2^2}z & \wtil d i & 0 \\
0 & -fi & \frac{r_1^2}{r_1^2+r_2^2}z & -\wtil\beta \\
\wtil f i & 0 & -\beta & -\frac{r_2^2}{r_1^2+r_2^2}z
\end{pmatrix},
\end{align}
with the notations in \eqref{M1}.
\end{lemma}

\begin{proof}
This is a routine calculation, that follows similarly as in the above cited
references \cite{Delvaux2012,DKZ,DG}. From the asymptotics
\eqref{M:asymptotics} we have for $z\to\infty$ that
\begin{multline}\label{diffeq1:RH}
\frac{\pp M}{\pp z}M^{-1} = \left(I+\frac{M_1}{z}+\cdots\right)
\begin{pmatrix} r_1^2\tau & 0 & i(r_1-s_1z^{-1}) & 0 \\ 0 & -r_2^2\tau & 0 & i(r_2+ s_2z^{-1}) \\
i(r_1z-s_1) & 0 & r_1^2\tau & 0 \\
0 & -i(r_2 z+ s_2) & 0 & -r_2^2\tau
\end{pmatrix}\\
\left(I-\frac{M_1}{z}+\cdots\right)+O(z^{-1}).
\end{multline}
Since the RH matrix $M(z)$ has constant jumps, the left hand side of
\eqref{diffeq1:RH} is an entire function of $z$. Liouville's theorem implies
that it is a polynomial in $z$. Collecting the polynomial terms in the right
hand side of \eqref{diffeq1:RH} we obtain
$$ U = \begin{pmatrix} r_1^2\tau & 0 & i r_1 & 0 \\ 0 & -r_2^2\tau & 0 & i r_2 \\
i(r_1z-s_1) & 0 & r_1^2\tau & 0 \\
0 & -i( r_2z+ s_2) & 0 & -r_2^2\tau
\end{pmatrix}+i M_1A- iA M_1,\qquad
A:=\begin{pmatrix}0&0&0&0\\ 0&0&0&0\\r_1&0&0&0\\0&-r_2&0&0\end{pmatrix}.$$ With
the help of \eqref{M1} and a small calculation, we then get
\eqref{Lax:zeta:nonsymm}. To obtain the $(3,1)$ and $(4,2)$ entries of
\eqref{Lax:zeta:nonsymm}, we also need the relations
\begin{eqnarray}\label{Lax:rel:nonsymm}
a+\alpha &=& -c^2 + \frac{r_2}{r_1} d\wtil d +\frac{s_1}{r_1}, \\
\wtil a+\wtil\alpha &=& -\wtil c^2 + \frac{r_1}{r_2} d\wtil d +\frac{s_2}{r_2},
\end{eqnarray}
which follow from the fact that the $(1,3)$ and $(2,4)$ entries in the $z^{-1}$
coefficient in \eqref{diffeq1:RH} are equal to zero. A similar argument yields
\eqref{Lax:s:nonsymm}--\eqref{Lax:tau:nonsymm}.
\end{proof}

\subsection{Proof of Proposition~\ref{prop:diffeq:pee}}
\label{subsection:proof:diffeqpee}

Let the vector $\mm(z)$ be a solution of $\frac{\pp}{\pp z}\mm=U\mm$ with $U$
in \eqref{Lax:zeta:nonsymm}. By splitting this equation in $2\times 2$ blocks
we get
\begin{equation}\label{split:2x2:1}
\begin{pmatrix} r_1^{-1}m_1'(z) \\ r_2^{-1}m_2'(z)\end{pmatrix}=\begin{pmatrix} -c+r_1\tau & r_1^{-1}r_2 d \\
-r_1r_2^{-1} \wtil d & \wtil c-r_2\tau
\end{pmatrix}\begin{pmatrix} m_1(z) \\ m_2(z)\end{pmatrix}+
i\begin{pmatrix} m_3(z) \\ m_4(z)\end{pmatrix},\qquad\qquad\qquad\qquad
\end{equation}\vspace{-6mm}
\begin{multline}\label{split:2x2:2}
\begin{pmatrix} r_1^{-1}m_3'(z) \\ r_2^{-1}m_4'(z)\end{pmatrix}=i\begin{pmatrix} c^2-r_1^{-1}r_2 d\wtil d-2r_1^{-1}s_1+z & -b-r_1^{-1}r_2\wtil\beta\\
-r_1r_2^{-1}\beta-\wtil b & \wtil c^2-r_1r_2^{-1} d\wtil d-2r_2^{-1}s_2-z
\end{pmatrix}\begin{pmatrix} m_1(z) \\ m_2(z)\end{pmatrix}\\ +
\begin{pmatrix} c+r_1\tau & d  \\
-\wtil d & -\wtil c-r_2\tau
\end{pmatrix}\begin{pmatrix} m_3(z) \\ m_4(z)\end{pmatrix}.
\end{multline}
From \eqref{split:2x2:1} and
\eqref{d:Painleve2:nonsymm}--\eqref{ctil:Hamiltonian:nonsymm}, we easily get
\eqref{diffeq:pee:3}--\eqref{diffeq:pee:4}. To prove the two remaining differential equations, we
take the derivative of \eqref{split:2x2:1} and use \eqref{split:2x2:2} to get
\begin{multline*}  \begin{pmatrix} r_1^{-2}m_1''(z) \\ r_2^{-2}m_2''(z)\end{pmatrix} =
\begin{pmatrix} -c+r_1\tau & r_1^{-2}r_2^2 d \\
-r_1^2r_2^{-2} \wtil d & \wtil c-r_2\tau
\end{pmatrix}\begin{pmatrix} r_1^{-1}m_1'(z) \\ r_2^{-1}m_2'(z)\end{pmatrix}
\\ - \begin{pmatrix} c^2-r_1^{-1}r_2 d\wtil d-2r_1^{-1}s_1+z & -b-r_1^{-1}r_2\wtil\beta\\
-r_1r_2^{-1}\beta-\wtil b & \wtil c^2-r_1r_2^{-1} d\wtil d-2r_2^{-1}s_2-z
\end{pmatrix}
\begin{pmatrix} m_1(z) \\ m_2(z)\end{pmatrix}\\
+\begin{pmatrix} c+r_1\tau & d  \\
-\wtil d & -\wtil c-r_2\tau
\end{pmatrix}
\left[\begin{pmatrix} r_1^{-1}m_1'(z) \\ r_2^{-1}m_2'(z)\end{pmatrix} -
\begin{pmatrix} -c+r_1\tau & r_1^{-1}r_2 d \\
-r_1r_2^{-1} \wtil d & \wtil c-r_2\tau
\end{pmatrix}\begin{pmatrix} m_1(z) \\ m_2(z)\end{pmatrix}\right].
\end{multline*}
From this equation and \eqref{d:Painleve2:nonsymm}--\eqref{betatil:PII} we
obtain the desired differential equations
\eqref{diffeq:pee:1}--\eqref{diffeq:pee:2}. This proves
Proposition~\ref{prop:diffeq:pee}(a).

To prove Part (b), let $\mm(z)$ satisfy the differential equations
\eqref{diffeq:pee:1}--\eqref{diffeq:pee:4}. From the proof of Part (a) above,
we see that $\frac{\pp}{\pp z}\mm=U\mm$ with $U$ in \eqref{Lax:zeta:nonsymm}.
But then $$\frac{\pp}{\pp z}[\what M^{-1} \mm]=\what M^{-1}(U-U)\mm=0,$$ which
implies Proposition~\ref{prop:diffeq:pee}(b). $\bol$

\subsection{Proof of Theorem~\ref{theorem:M1:nonsymm}} \label{subsection:proofPII}

The matrices $U,V,W$ in \eqref{Lax:system} satisfy the compatibility conditions
\begin{align}\label{compat:UV}
 \frac{\pp U}{\pp s} &=
\frac{\pp V}{\pp z} - UV + VU\\
\label{compat:UW}
 \frac{\pp U}{\pp \tau}&=
\frac{\pp W}{\pp z} - UW + WU.
\end{align}
These relations are obtained by calculating the mixed derivatives
$\frac{\pp^2}{\pp z\pp s}M=\frac{\pp^2}{\pp s\pp z}M$ and $\frac{\pp^2}{\pp
z\pp \tau}M=\frac{\pp^2}{\pp \tau\pp z}M$, respectively, in two different ways.

\begin{lemma}\label{lemma:DG:1}
Consider the matrices $U,V,W$ in
\eqref{Lax:zeta:nonsymm}--\eqref{Lax:tau:nonsymm}. Then with the expressions
\eqref{d:Painleve2:nonsymm}--\eqref{ftil:PII} in
Theorem~\ref{theorem:M1:nonsymm}, the compatibility conditions
\eqref{compat:UV}--\eqref{compat:UW} are satisfied.
\end{lemma}

\begin{proof}
This is a lengthy but direct calculation. It is best performed with the help of
a symbolic computer program such as Maple. First we consider the compatibility
condition with respect to $s$. Writing the matrix equation \eqref{compat:UV} in
entrywise form, with the help of Maple, we obtain the system of equations, with
the prime denoting the derivative with respect to $s$,
\begin{eqnarray}
\label{comp:s:d1}r_1 d' &=& 2(\si_1 r_2+\si_2 r_1)(\wtil c d-b)+2(r_1^2+r_2^2)\si_1\tau d \\
\label{comp:s:d2}r_2 d' &=& 2(\si_1 r_2+\si_2 r_1)(c d-\wtil\beta)-2(r_1^2+r_2^2)\si_2\tau d \\
\label{comp:s:dtil1}r_1\wtil d' &=& 2(\si_1 r_2+\si_2 r_1)(\wtil c \wtil d-\beta)-2(r_1^2+r_2^2)\si_1\tau \wtil d \\
\label{comp:s:dtil2}r_2\wtil d' &=& 2(\si_1 r_2+\si_2 r_1)(c \wtil d-\wtil b)+2(r_1^2+r_2^2)\si_2\tau \wtil d \\
\label{comp:s:c} r_1 c' &=& 2(\si_1 r_2+\si_2 r_1)d\wtil d+2\si_1^2 s \\
\label{comp:s:ctil} r_2\wtil c' &=& 2(\si_1 r_2+\si_2 r_1)d\wtil d+2\si_2^2 s
\end{eqnarray}
and
\begin{multline}\label{comp:s:bbetatil}(r_1 b+r_2\wtil\beta)' = (r_1\wtil c+r_2 c)d'-2(r_1^2+r_2^2)\tau\si_1
(\wtil c
d-b)+2(r_1^2+r_2^2)\tau\si_2(cd-\wtil\beta)\\
-2\frac{r_1\si_2+r_2\si_1}{r_1r_2}\left((r_1^2+r_2^2)d^2\wtil
d+(r_1\si_2+r_2\si_1)sd\right)-4\si_1 \si_2 sd
\end{multline}
\vspace{-8mm}
\begin{multline}(r_1\beta+r_2\wtil b)' = (r_1\wtil c+r_2 c)\wtil d'+2(r_1^2+r_2^2)\si_1\tau (\wtil c
\wtil d-\beta)-2(r_1^2+r_2^2)\si_2\tau(c\wtil d-\wtil b)\\
-2\frac{r_1\si_2+r_2\si_1}{r_1r_2}\left((r_1^2+r_2^2)\wtil d^2
d+(r_1\si_2+r_2\si_1)s\wtil d\right)-4\si_1 \si_2 s\wtil d.
\end{multline}

Next we consider the compatibility condition with respect to $\tau$. By writing
the matrix equation \eqref{compat:UW} in entrywise form, with the help of
Maple, we obtain, with the prime denoting the derivative with respect to
$\tau$,
\begin{eqnarray}
\label{comp:tau:d1} r_1 d' &=&
(r_1^2+r_2^2)(r_1^2\tau\wtil\beta+r_2^2\tau\wtil\beta+r_1 c\wtil\beta +r_2\wtil c\wtil \beta +r_2  d^2 \wtil d-r_1 c^2 d+2 s_1 d+r_2 f)\\
\label{comp:tau:d2} r_2 d' &=&
(r_1^2+r_2^2)(r_1^2\tau b+r_2^2\tau b -r_1 c b-r_2 \wtil c b -r_1 d^2\wtil d+r_2 \wtil c^2 d -2s_2 d-r_1 f)\\
\label{comp:tau:d3} r_1 \wtil d' &=&
(r_1^2+r_2^2)(r_1^2\tau\wtil b +r_2^2\tau\wtil b-r_2 \wtil c\wtil b-r_1 c\wtil b-r_2\wtil d^2 d+r_1 c^2 \wtil d-2s_1\wtil d -r_2\wtil f)\\
\label{comp:tau:d4} r_2 \wtil d' &=& (r_1^2+r_2^2) (r_1^2\tau\beta
+r_2^2\tau\beta+r_2\wtil c\beta+r_1 c\beta
 +r_1\wtil d^2 d-r_2\wtil c^2\wtil d+2 s_2\wtil d+r_1\wtil f)\\
\label{comp:tau:c} c' &=& (r_1^2+r_2^2)(d\beta-\wtil d b)  \\
\label{comp:tau:ctil} \wtil c' &=& (r_1^2+r_2^2)(\wtil d\wtil\beta- d\wtil
b)\end{eqnarray} and \begin{multline*} r_1b'+r_2\wtil\beta' =
(r_1^2+r_2^2)(-r_1 c^2 b+r_2\wtil
c^2\wtil\beta-r_1 d\wtil d\wtil\beta+r_2 d \wtil d b+2 s_1 b-2s_2\wtil b\\
-r_1^2\tau f-r_2^2 \tau f-r_1 c f+r_2 \wtil c
f)\end{multline*}\vspace{-8mm}\begin{multline*} r_2\wtil b'+r_1\beta' =
(r_1^2+r_2^2)(-r_2 \wtil c^2 \wtil b+r_1
c^2 \beta-r_2 d\wtil d\beta+r_1 d \wtil d \wtil b+2 s_2 \wtil b-2s_1 b\\
-r_2^2\tau \wtil f-r_1^2 \tau \wtil f-r_2 \wtil c \wtil f+r_1 c \wtil f).
\end{multline*}
Direct calculations show that all these equations are satisfied by
\eqref{d:Painleve2:nonsymm}--\eqref{ftil:PII}.
\end{proof}

\begin{lemma}\label{lemma:DG:2}
Theorem~\ref{theorem:M1:nonsymm} holds true if $\tau=0$.
\end{lemma}

\begin{proof} Equations \eqref{d:Painleve2:nonsymm}--\eqref{ctil:Hamiltonian:nonsymm} follow from
\cite[Th.~2.4]{DKZ}. The other equations in Theorem~\ref{theorem:M1:nonsymm}
then follow from \eqref{comp:s:d1}--\eqref{comp:s:dtil2} and
\eqref{comp:tau:d2}--\eqref{comp:tau:d3}.
\end{proof}

With the help of Lemmas~\ref{lemma:DG:1}--\ref{lemma:DG:2}, one can prove
Theorem~\ref{theorem:M1:nonsymm} for $\tau\neq 0$ in the same way as in
\cite[Sec.~5]{DG}, where the symmetric case was considered. This is a lengthy
and tedious calculation that follows exactly the same plan as in \cite{DG}. We
note that the same reasoning also yields the solvability statement in
Proposition~\ref{prop:solv}. We do not go into the details.

\subsection*{Alternative approach to Theorem~\ref{theorem:M1:nonsymm}}

The above reasoning does not give any insight on the origin of the expressions
in Theorem~\ref{theorem:M1:nonsymm}. Therefore, in the remaining part of this
section, let us deduce these formulas in a more direct way. The calculations
below are partly heuristic in the sense that we will make an ansatz
\eqref{comp:ansatz:s}, \eqref{comp:ansatz:tau}. We start with

\begin{lemma} The numbers $d,\wtil d$ in \eqref{M1} satisfy the system of coupled second-order
differential equations
\begin{multline}\label{comp:d:PII}
\frac{\pp^2 d}{\pp s^2} = 4\tau(r_1\si_1-r_2\si_2)\frac{\pp d}{\pp s}
\\ -4(r_1^2+r_2^2)(\si_1^2+\si_2^2) \tau^2 d+8\frac{(r_1\si_2+r_2\si_1)^2}{r_1 r_2} d^2\wtil
d+8\frac{(r_1\si_2+r_2\si_1)^3}{r_1r_2(r_1^2+r_2^2)}sd,
\end{multline}
\vspace{-8mm}
\begin{multline}\label{comp:dtil:PII}
\frac{\pp^2 \wtil d}{\pp s^2} = -4\tau(r_1\si_1-r_2\si_2)\frac{\pp\wtil d}{\pp
s}
\\ -4(r_1^2+r_2^2)(\si_1^2+\si_2^2) \tau^2 \wtil d+8\frac{(r_1\si_2+r_2\si_1)^2}{r_1 r_2} \wtil d^2
d+8\frac{(r_1\si_2+r_2\si_1)^3}{r_1r_2(r_1^2+r_2^2)}s\wtil d.
\end{multline}
\vspace{-3mm} Moreover,
\begin{equation}\label{comp:d:mixed} \frac{\pp
d}{\pp\tau}= -\frac{r_1r_2(r_1^2+r_2^2)}{\si_1r_2+\si_2 r_1}\tau\frac{\pp
d}{\pp s} + (r_1^2+r_2^2)^2 \frac{\si_1 r_2-\si_2 r_1}{\si_1r_2+\si_2
r_1}\tau^2 d+2(r_1 s_1- r_2 s_2) d.
\end{equation}
\end{lemma}

\begin{proof} Equation~\eqref{comp:d:PII} follows from \eqref{comp:s:d1}--\eqref{comp:s:d2} and \eqref{comp:s:c}--\eqref{comp:s:bbetatil}
after some lengthy algebraic manipulations. Equation~\eqref{comp:dtil:PII}
follows by symmetry. To obtain \eqref{comp:d:mixed}, first note that
\eqref{comp:s:d1}--\eqref{comp:s:d2} imply the relations
\begin{equation}\label{comp:bbs}
  r_2(\wtil c d-b)-r_1(cd-\wtil\beta)+(r_1^2+r_2^2)\tau d=0,
\end{equation}
\begin{equation}\label{comp:ddcomb}
r_1 r_2 \frac{\pp d}{\pp s} = (\si_1 r_2+\si_2 r_1)(r_2\wtil c d-r_2b+r_1c
d-r_1\wtil\beta)+(r_1^2+r_2^2)(\si_1 r_2-\si_2 r_1)\tau d.
\end{equation}
Now by eliminating $f$ from \eqref{comp:tau:d1}--\eqref{comp:tau:d2} we get
\begin{equation}
\frac{\pp d}{\pp\tau} = (r_1^2+r_2^2) (r_2 b + r_1\wtil\beta) \tau-(r_2
b-r_1\wtil\beta) (r_1 c+r_2 \wtil c) -d (r_1^2 c^2-r_2^2 \wtil c^2)+2(r_1 s_1-
r_2 s_2) d,
\end{equation}
On account of \eqref{comp:bbs} this becomes
\begin{eqnarray*}
\frac{\pp d}{\pp\tau} &=& (r_1^2+r_2^2) (r_2 b + r_1\wtil\beta-r_1 cd-r_2 \wtil
c d)\tau +2(r_1 s_1- r_2 s_2) d.
\end{eqnarray*}
Combining this with \eqref{comp:ddcomb} we obtain \eqref{comp:d:mixed}.
\end{proof}

We seek a solution to the differential equations
\eqref{comp:d:PII}--\eqref{comp:dtil:PII} in the form
\begin{equation}\label{comp:ansatz:0} d = e^{h}g,\qquad \wtil d = e^{-h} g,
\end{equation}
with $h=h(s,\tau)$ an odd function of $\tau$ and $g=g(s,\tau)$ an even function
of $\tau$ (recall Lemma~\ref{lemma:symmetry:M1}). Plugging this in
\eqref{comp:d:PII} we find, with again the prime denoting the derivative with
respect to $s$,
\begin{multline}\label{comp:g:PII}
g''+2h'g'+((h')^2+h'') g = 4\tau(r_1\si_1-r_2\si_2)(g'+h' g)
\\ -4(r_1^2+r_2^2)\tau^2 (\si_1^2+\si_2^2) g+8\frac{(r_1\si_2+r_2\si_1)^2}{r_1 r_2}
g^3+8\frac{(r_1\si_2+r_2\si_1)^3}{r_1r_2(r_1^2+r_2^2)}sg.
\end{multline}
To obtain further progress we make the ansatz
\begin{equation}\label{comp:ansatz:s} \frac{\pp^2 h}{\pp s^2}=0,\qquad \frac{\pp h}{\pp s} =
2(r_1\si_1-r_2\si_2)\tau. \end{equation} After a little calculation,
\eqref{comp:g:PII} then simplifies to
\begin{equation}\label{comp:g:PII:2}
g'' =
\\ -4(r_1\si_2+r_2\si_1)^2\tau^2 g+8\frac{(r_1\si_2+r_2\si_1)^2}{r_1r_2}
g^3+8\frac{(r_1\si_2+r_2\si_1)^3}{r_1r_2(r_1^2+r_2^2)}sg.
\end{equation}
We can relate \eqref{comp:g:PII:2} to the Painlev\' e~II equation. We have that
$q=q(s)$ satisfies $q'' = s q + 2q^3$, if and only if
\begin{equation}\label{comp:ansatz:PII} g(s) := c_1 q(c_2 s + c_3) \end{equation} satisfies
\begin{equation}
    g'' = c_2^2 c_3 g+2 \frac{c_2^2}{c_1^2} g^3+ c_2^3 s g.
\end{equation}
Comparing coefficients with \eqref{comp:g:PII:2}, we see that \begin{eqnarray}
\label{comp:PII:1}c_1 &=&  \frac{(r_1r_2)^{1/6}}{(r_1^2+r_2^2)^{1/3}} \\
\label{comp:PII:2}c_2 &=&
2\frac{(r_1\si_2+r_2\si_1)}{(r_1r_2(r_1^2+r_2^2))^{1/3}}  \\
\label{comp:PII:3}c_3 &=& -(r_1r_2(r_1^2+r_2^2))^{2/3}\tau^2.
\end{eqnarray}
Finally, substituting the formulas \eqref{comp:ansatz:0},
\eqref{comp:ansatz:PII},
\begin{equation}\label{comp:ansatz} d = e^{h}g =e^h \frac{(r_1r_2)^{1/6}}{(r_1^2+r_2^2)^{1/3}}
q\left(2\frac{(r_1\si_2+r_2\si_1)}{(r_1r_2(r_1^2+r_2^2))^{1/3}} s
-(r_1r_2(r_1^2+r_2^2))^{2/3}\tau^2\right)
\end{equation}
in \eqref{comp:d:mixed} and using again \eqref{comp:ansatz:s} we find after
some calculations,
\begin{equation}\label{comp:ansatz:tau} \frac{\pp h}{\pp \tau} =
(r_2^4-r_1^4)\tau^2+2(r_1\si_1-r_2\si_2)s,
\end{equation}
where we are assuming that the choice of the Painlev\'e~II solution $q$ in
\eqref{comp:ansatz} is independent from $\tau$. From
\eqref{comp:ansatz}--\eqref{comp:ansatz:tau} and the known result for $\tau=0$
\cite[Th.~2.4]{DKZ} we get the expression for $d$ in
Theorem~\ref{theorem:Airyformulas:FV} (with $q$ the Hastings-McLeod solution to
Painlev\'e~II). By symmetry we obtain the expression for $\wtil d$. From
\eqref{comp:tau:c}, \eqref{comp:s:d1}, \eqref{comp:s:dtil1} and a little
calculation we then find $$\frac{\pp}{\pp\tau}c = -2r_1^{-1}(r_1^2+r_2^2)\tau
C^{-2} q^2(\si)= \frac{\pp}{\pp\tau}(-r_1^{-1}C^{-1} u(\si)),$$ where the
second equality follows from \eqref{TW:rel3} and \eqref{C:RH}. Combining this
with the known result for $\tau=0$ \cite[Th.~2.4]{DKZ} we get the expression
for $c$ in Theorem~\ref{theorem:Airyformulas:FV}.
From \eqref{comp:s:d1} and a little calculation we find the expression for $b$,
while \eqref{comp:tau:d2} yields the formula for $f$. Finally, the remaining
formulas in Theorem~\ref{theorem:Airyformulas:FV} follow from symmetry
considerations, see Lemma~\ref{lemma:symmetry:M1}. $\bol$

\end{document}